\documentclass{article}

\usepackage{amsmath,amsfonts,euscript,amssymb,amsthm,a4,times,color}

\usepackage[dvipsnames]{xcolor}

\usepackage{amssymb,amsmath,amsfonts,amsthm,latexsym,enumerate, esint, manyfoot}
\usepackage{graphics,graphicx}
\usepackage[english]{babel}

\usepackage[colorlinks,linkcolor = blue,citecolor = red,pagebackref=false]{hyperref}

\newcommand{\mollifytime}[2]{[\![ #1 ]\!]_{#2}}

\newcommand{\N}{\mathbb N}

\newcommand{\eps}{\varepsilon}

\newcommand{\wto}{\rightharpoondown}
\newcommand{\wsto}{\overset{\raisebox{-1ex}{\scriptsize $*$}}{\rightharpoondown}}

\newcommand{\dive}{\mathrm{div}\,}

\def\real{\mathbb R}

\def\N{\mathbb N}

\def\R{\mathbb R}

\def\proofof#1{\begin{proof}[Proof of #1]}

\def\d{\mathrm{d}}

\def\dx{\d x}

\def\dt{\d t}
\def\dz{\d z}

\def\={^{\wedge}}

\def\eqn#1$$#2$${\begin{equation}\label#1#2\end{equation}}

\def\mir1{\mathcal L_1}

\DeclareMathOperator\spt{spt}
\DeclareMathOperator*\esssup{ess\ sup}
\DeclareMathOperator\dist{dist}

\newtheorem{thm}{Theorem}[section]

\newtheorem{lem}[thm]{Lemma}

\newtheorem{prop}[thm]{Proposition}

\theoremstyle{definition}
\newtheorem{defn}[thm]{Definition}
\newtheorem{rem}[thm]{Remark}

\newcommand{\xint}[3]{{\setbox0=\hbox{$#1{#2#3}{\int}$}
   \vcenter{\hbox{$#2#3$}}\kern-.5\wd0}}
\newcommand{\mint}{\mathchoice
   {\xint\displaystyle\textstyle-}
   {\xint\textstyle\scriptstyle-}
   {\xint\scriptstyle\scriptscriptstyle-}
   {\xint\scriptscriptstyle\scriptscriptstyle-}
   \!\int}

\makeatletter \catcode`@=11
\newbox\tr@tto
\setbox\tr@tto=\hbox{{\count0=0\dimen0=-,9pt\dimen1=1,1pt\loop\ifnum
    \count0<11 \advance \count0 by1 \vrule width.51pt height\dimen1
    depth\dimen0\kern-0.17pt\advance\dimen0 by-0.05pt\advance\dimen1
    by0.1pt\repeat \loop\ifnum\count0<21\advance \count0 by1 \vrule
    width.6pt height\dimen1 depth\dimen0\kern-0.2pt \advance\dimen0
    by-0.1pt\advance\dimen1 by 0.05pt\repeat}}
\def\medint{\displaystyle\copy\tr@tto\kern-10.4pt\int}
\numberwithin{equation}{section}

\allowdisplaybreaks

\usepackage{amsmath,amsfonts,euscript,amssymb,amsthm,a4,times,color}

\usepackage[dvipsnames]{xcolor}
\usepackage{ulem}

\usepackage{amssymb,amsmath,amsfonts,amsthm,latexsym,enumerate, esint, manyfoot}
\usepackage{graphics,graphicx}
\usepackage[english]{babel}

\def\real{\mathbb R}

\def\N{\mathbb N}

\def\R{\mathbb R}

\def\proofof#1{\begin{proof}[Proof of #1]}

\def\d{\mathrm{d}}

\def\A{\mathcal{A}}

\def\dx{\d x}

\def\dt{\d t}

\def\={^{\wedge}}

\def\eqn#1$$#2$${\begin{equation}\label#1#2\end{equation}}

\def\mir1{\mathcal L_1}

\def\AA{\mathbf{A}}
\def\BB{\mathbf{B}}

\makeatletter \catcode`@=11
\newbox\tr@tto
\setbox\tr@tto=\hbox{{\count0=0\dimen0=-,9pt\dimen1=1,1pt\loop\ifnum
    \count0<11 \advance \count0 by1 \vrule width.51pt height\dimen1
    depth\dimen0\kern-0.17pt\advance\dimen0 by-0.05pt\advance\dimen1
    by0.1pt\repeat \loop\ifnum\count0<21\advance \count0 by1 \vrule
    width.6pt height\dimen1 depth\dimen0\kern-0.2pt \advance\dimen0
    by-0.1pt\advance\dimen1 by 0.05pt\repeat}}
\def\medint{\displaystyle\copy\tr@tto\kern-10.4pt\int}
\numberwithin{equation}{section}

\allowdisplaybreaks

\begin{document}

\title{Local boundedness for solutions to \\parabolic $p,q$-problems with degenerate coefficients}

\author{ Flavia Giannetti - Antonia Passarelli di Napoli - Christoph Scheven \\ \\
\normalsize{Dipartimento di Matematica e Applicazioni "R.
Caccioppoli"} \\ \normalsize{Universit\`{a} di Napoli ``Federico
II", via Cintia - 80126 Napoli}\\ \\ \normalsize{Faculty for Mathematics - University Duisburg-Essen}\\
\normalsize{Thea-Leymann-Str. 9, D-45127 Essen, Germany}
 \\ \normalsize{e-mail:
giannett@unina.it, antpassa@unina.it, christoph.scheven@uni-due.de}}

\bigskip

\maketitle

\medskip

\begin{abstract}\noindent
We investigate the local boundedness of solutions $u:\Omega_T\to\R$ to parabolic equations of the form
\begin{equation*}
  \partial_tu-\dive \A(x,t,Du)=0 \qquad\mbox{in }\Omega_T=\Omega\times(0,T)
\end{equation*}
that satisfy $p,q$-growth conditions and have degenerate coefficients. More precisely, we assume structure conditions of the type
\begin{align*}
|\mathcal{A}(x,t,\xi)|&\le b(x,t)(\mu^2+|\xi|^2)^{\frac{q-1}{2}},\\
\langle \A(x,t,\xi),\xi\rangle&\ge a(x,t)(\mu^2+|\xi|^2)^{\frac {p-2}{2}}|\xi|^2,
\end{align*}
for $2\le p\le q$ and $\mu\in[0,1]$, where the functions $a^{-1}, b:\Omega_T\to\R$ are possibly unbounded and only satisfy some integrability condition. Under a certain assumption on the gap between $p$ and $q$, we prove two main results. First, we show that subsolutions that are contained in the natural energy space are locally bounded from above. Second, for parabolic equations with a variational structure, we use these bounds to show the existence of locally bounded variational solutions.
\end{abstract}



\bigskip
\bigskip
\bigskip


\section{Introduction}
We consider parabolic equations of the type
\begin{equation}\label{equa0}
  \partial_tu-\dive \A(x,t,Du)=0 \qquad\mbox{in }\Omega_T=\Omega\times(0,T),
\end{equation}
where $\Omega$ is an open bounded set in $\real^n$, $n\ge2$, and $T>0$. 
The Carath\'eodory function $\mathcal{A}(x,t,\xi):\Omega_T\times\R^{n}\to\R^{n}$ is  such that
 for almost every $(x,t)\in\Omega_T$ and all $\xi,\eta\in
\real^{n}$, the following  conditions are satisfied
\begin{align}\label{ip1}
|\mathcal{A}(x,t,\xi)|&\le b(x,t)(\mu^2+|\xi|^2)^{\frac{q-1}{2}},\\[0.8ex]
\label{ip2}
\langle \A(x,t,\xi),\xi\rangle&\ge a(x,t)(\mu^2+|\xi|^2)^{\frac {p-2}{2}}|\xi|^2,
\end{align}
for exponents $2\le p\le q$ and a parameter  $\mu\in[0,1]$. The functions
$a:\Omega_T\to[0,\infty)$ and $b:\Omega_T\to[0,\infty)$ satisfy 
\begin{equation}\label{summab}
\frac{1}{a(x,t)}\in L^\alpha(\Omega_T)\quad\text{and}\quad b(x,t)\in L^\beta(\Omega_T),
\end{equation}
for a couple of exponents $\alpha,\beta>1$.
A simple model case of the systems that we have in mind is given by
\begin{equation*}
  \partial_tu-\dive\Big(a(x,t)|Du|^{p-2}Du+b(x,t)|Du|^{q-2}Du\Big)=0
  \qquad\mbox{in $\Omega_T$},
\end{equation*}
with $a(x,t)\ge0$ possibly zero 
and $b(x,t)\ge0$ unbounded in $\Omega_T$ and  $2\le p\le q$. 

The elliptic counterpart  of equation \eqref{equa0} reads as
\begin{equation}\label{uno}
  \dive \A(x,Du)=0 \qquad\mbox{in }\Omega,
\end{equation}
with  \begin{equation}\label{unobis}\lambda(x)(\mu^2+|\xi|^2)^{\frac{p-2}{2}}|\xi|^2\le\langle \A(x,\xi),\xi\rangle \le \Lambda(x)(\mu^2+|\xi|^2)^{\frac{q-2}{2}}|\xi|^2\end{equation}
for a.e. $x\in \Omega$  and  for every $\xi\in \R^n$. This condition combines the so-called $p,q$-growth conditions with degenerate or singular coefficients. More precisely, the common terminology refers to equation  \eqref{uno} as degenerate if   $\lambda^{-1}$ is unbounded and singular if $\Lambda$ is unbounded. 
Degenerate or singular equations in this sense were first considered by 
Trudinger \cite{Trudinger:1971} in the linear case, in which $p=q=2$ and $\mu=0$. He proved that any weak solution of 
\eqref{uno} is locally bounded in $\Omega $, under the following
integrability assumptions on $\lambda $ and $\Lambda $ 
\begin{equation}
\lambda ^{-1}\in L_{\mathrm{loc}}^{r}(\Omega )\quad \text{and}\quad \Lambda
_{1}=\lambda ^{-1}\Lambda ^{2}\in L_{\mathrm{loc}}^{\sigma }(\Omega )\quad \text{%
with $\frac{1}{r}+\frac{1}{\sigma }<\frac{2}{n}$}.
\label{risultato-trudinger}
\end{equation}%
The result by Trudinger was extended in many settings and directions:
firstly, by Trudinger himself in \cite{trud2} and later by Fabes, Kenig and
Serapioni in \cite{fabes-kenig-serapioni}; Pingen in \cite{pingen} dealt
with systems. 
Very recently, condition~\eqref{risultato-trudinger} has been improved to
\begin{equation*}
\frac{1}{r}+\frac{1}{\sigma }<\frac{2}{n-1}
\end{equation*}
by Bella and Sch\"affner \cite{Bella-Schaeffner:2021}, see also \cite{Bella-Schaeffner:2023, Hirsch-Schaeffner} for corresponding results in the nonlinear case. 
Related results up to the boundary are contained in \cite{DeFilippis-Piccinini}. 
The literature
concerning non-uniformly elliptic problems is extensive and we refer the
interested reader to the references in \cite{bdgp}. 

Another field of research that has been intensively studied in the literature are elliptic equations or functionals with $p,q$-growth, which corresponds to~\eqref{unobis} with bounded functions $\lambda^{-1},\Lambda$, but now $p<q$. 
The study of regularity in the $p,q$-growth context started
with  Marcellini \cite{mar89, mar91} and, since then,
many papers devoted  to this subject have been provided. Here, we just refer to   \cite{DeFilippis-Mingione1,DeFilippis-Mingione2,mar20-2,dark-survey,Mingione-Radulescu} and the references therein.

More recently, the $p,q$-growth condition has been combined  with degenerate or singular coefficients in \cite{CMM, CMMP}. In these papers, the authors proved the local boundedness and later on the Lipschitz regularity of  minimizers of integral functionals whose Euler-Lagrange equations are of type \eqref{uno} under the general condition \eqref{unobis}.

In the parabolic setting, the regularity theory for problems with non-standard growth is far less complete. Gradient bounds for solutions to parabolic equations  with $p,q$-growth have been obtained in \cite{BDM:pq-equations,singer1,DeFilippis}, and \cite{GPS} contains a result on higher differentiability. Local boundedness for gradient flows with $p,q$-growth has been established in \cite{singer2}. Existence of solutions to such gradient flows, also in the vector-valued case, has been proved via variational methods in \cite{BDM:pq,BDM:variational,singer3}.

To our knowledge, no regularity results seem to be available for solutions of  degenerate parabolic problems of the type \eqref{equa0} under the assumptions \eqref{ip1}, \eqref{ip2} and \eqref{summab}, even in the case $q=p$.
The aim of this paper is to give a first contribution to the study of this kind of problems, by investigating the local boundedness of the  solutions.

It is well known that to get regularity under $p,q$-growth the exponents $q$
and $p$ cannot be too far apart; usually, in the uniformly elliptic case, the gap between $p$ and $q$ is
described by a condition relating $p,q$ and the dimension $n$. In our case
we take into account the possible degeneracy of $a(x,t)$ and $b(x,t)$, so that the gap will depend
on $\alpha$, the summability exponent of $a^{-1}$ that \textquotedblleft
measures\textquotedblright\ how degenerate $a$ is, and the exponent $\beta$
that tells us how far $b(x,t)$ is from being bounded. More precisely, our results rely on the assumption 
\begin{equation}\label{gap}
p\le q<p\frac{\alpha}{\alpha+1}\frac{\beta-1}{\beta}+\frac{2}{n+2}.
\end{equation}
We note that this assumption implicitly contains a restriction on the possible values of $\alpha$ and $\beta$ of the form 
\begin{equation}
    \label{condition-alpha-beta}
    \frac{\alpha}{\alpha+1}\frac{\beta-1}{\beta} >1-\frac{2}{p(n+2)}\ge 1-\frac{1}{n+2}.
\end{equation} 

If $\alpha=\beta=\infty $ then %
\eqref{gap} reduces to $\frac{q}{p}<1+\frac{2}{p(n+2)}$ that is reminiscent of the bound obtained in
 \cite{GPS}. Note that the bound on the ratio $\frac{q}{p}$ here depends on  $n+2$ instead of $n$, as it happens in the elliptic setting, due to the different scaling in time.
 Our main results are contained in the following two subsections.

\subsection{Local bounds for weak subsolutions}

Our first main result is a local upper bound for weak subsolutions
that are contained in the natural energy space.  
In the sequel, we use the abbreviations 
\begin{equation}\label{def-p-alpha}
    p_\alpha:=\frac{p\alpha}{\alpha+1}
    \qquad\mbox{and}\qquad
    q_\beta:=\frac{q\beta}{\beta-1}.
\end{equation}

We consider weak subsolutions in the following sense. 
\begin{defn}\label{def:weak-solution}
A function $u\in L^{q_\beta}(0,T;W^{1,q_\beta}(\Omega))\cap
C_w([0,T];L^2(\Omega))$ is a weak subsolution of~\eqref{equa0} if 
\begin{equation*}
    \int_{\Omega_T}\big(-u\partial_t\varphi
    +\langle \A(x,t,Du),D\varphi\rangle\big)\,\dx\dt\le 0
  \end{equation*}
  holds true for every nonnegative test function $\varphi\in C_0^\infty(\Omega_T,\R_{\ge0})$.
  \end{defn}

Here, we used the notation $C_w([0,T];L^2(\Omega))$ for the space
of maps $u\in L^\infty(0,T;L^2(\Omega))$ that are continuous with
respect to the weak $L^2$-topology.

\begin{thm}\label{thm:subsolution}
   Assume that $u\in L^{q_\beta}(0,T; W^{1,q_\beta}(\Omega))$ is a weak subsolution to~\eqref{equa0}, under assumptions~\eqref{ip1}, \eqref{ip2}, \eqref{summab} and \eqref{gap}.
   Then, for every cylinder $Q:=Q_{2\rho,2\sigma}(z_o)$ with
   $\dist(Q,\partial_{\mathrm{par}}\Omega_T)>0$ for parameters $\rho,\sigma>0$, we have the following estimate on the smaller cylinder $\frac12 Q:=Q_{\rho,\sigma}(z_o)$.
    \begin{align*}
        \esssup_{\frac12 Q}u
         \le
        \mathrm{C}_{a,b}\max\left\{\bigg(\frac{\sigma}{\rho^{\frac{p}{p+1-q}}}\bigg)^{\vartheta_1}\frac{1}{\rho^{(q-p)\vartheta_2}}\bigg(\mint_{Q} u_+^m\,\d z\bigg)^{\vartheta_3},\bigg(\mint_{Q} u_+^m\,\d z\bigg)^{\frac1m},\bigg(\frac{\rho^{\frac{p}{p+1-q}}}{\sigma}\bigg)^{\frac{p+1-q}{p-2+2(q-p)}},\rho\right\}.
    \end{align*} 
    Here, we abbreviated
    $m:=\frac{p_\alpha(n+2)}{n}$. The constants $\vartheta_i>0$, $i\in\{1,2,3\}$, depend at most on $n,p,q,\alpha,$ and $\beta$ and the constant $\mathrm{C}_{a,b}$ depends additionally on the quantities
$\mint_{Q}a^{-\alpha}\,\dx\dt$ and $\mint_{Q}b^{\beta}\,\dx\dt$.
\end{thm}

The proof of Theorem~\ref{thm:subsolution} is contained in Section~\ref{sec:subsolution}. In addition,
in Section~\ref{sec:subsolution} we consider slightly more general conditions that also
include solutions to regularized problems, see Proposition~\ref{prop:apriori}.

\begin{rem}\label{rem:sup-est}
  The estimate from the above theorem suggests that it is natural to consider parabolic cylinders of the form $Q=Q_{\rho,\sigma}(z_o)$ with $\sigma=\rho^{\frac{p}{p+1-q}}$. In fact, for a cylinder of this type with $\rho\le1$, the preceding estimate simplifies to 
  \begin{align*} 
    \esssup_{\frac12 Q}u
    \le
    \mathrm{C}_{a,b}\max\left\{\frac{1}{\rho^{(q-p)\vartheta_2}}\bigg(\mint_{Q} u_+^m\,\d z\bigg)^{\vartheta_3},\bigg(\mint_{Q} u_+^m\,\d z\bigg)^{\frac1m},1\right\}.
  \end{align*}
  Note that the right-hand side is finite due to Sobolev's embedding, cf. Lemma~\ref{lem:inter}.
\end{rem}

We conclude this subsection with a brief outline of the proof and some comments on the challenges that arise in the process.
The first step in the proof is to derive a Caccioppoli type inequality for the truncated functions $(u-k)_+$ with levels $k\in\R$, see Lemma~\ref{lem:caccioppoli}. Formally, this is done by testing the definition of the subsolution with the test function $(u-k)_+\Phi$ for a suitable cut-off function $\Phi\in C^\infty_0(\Omega_T)$. However, it is not possible to use this function directly as a test function, because the subsolution might not have the required regularity with respect to the time variable. Therefore, we have to use time mollifications of the subsolution to construct an admissible test function. When passing to the limit with the mollification parameter, we rely on the assumption $Du\in L^{q_\beta}(\Omega_T)$ to guarantee that the resulting integrals converge. In fact, it is not enough to assume that $b|Du|^q\in L^1(\Omega_T)$, as the growth assumption \eqref{ip1} suggests, because this property might not be preserved under the time mollification. This is the point in the proof at which we have to restrict ourselves to subsolutions in the energy space $L^{q_\beta}(0,T;W^{1,q_\beta}(\Omega))$. 

The resulting Caccioppoli inequality is only useful for our purposes as long as the gap between $p$ and $q$ is not too large. More precisely, the exponents of $(u-k)_+$ appearing on the right-hand side of the Caccioppoli inequality need to be strictly less than the integrability exponent $m=\frac{p_\alpha(n+2)}{n}$ provided by Sobolev's embedding, cf. Lemma~\ref{lem:inter}. This requirement leads to assumption~\eqref{gap} on the exponents $p$ and $q$, cf. Remark~\ref{rem:gap}.

Having established the Caccioppoli inequality, we prove the claimed local upper bound by setting up a De Giorgi iteration scheme. As usual in the case of nonlinear parabolic problems, a challenge arises from the different scaling behaviour of the time derivative and the diffusion term in the parabolic equation. This problem is already present in the case of the parabolic $p$-Laplace equation, see e.g. \cite{DiBenedettobook}. Under the much more general assumptions considered here, the situation becomes significantly more involved. On a technical level, we need to balance the various terms appearing on the right-hand side of the Caccioppoli inequality. This is accomplished by an intricate choice of the levels used in the De Giorgi iteration according to \eqref{choice-k}. This choice compensates for the inhomogeneous form of the Caccioppoli inequality and makes it possible to prove the claimed local upper bound for weak subsolutions.

\subsection{Existence of locally bounded variational solutions}

For our second main result, we apply the local bounds stated above
to prove existence of locally bounded
solutions to Cauchy-Dirichlet problems. For the existence result, we
restrict ourselves to parabolic equations with variational structure,
for which the concept of variational solutions is available. More
precisely, we consider Cauchy-Dirichlet problems of the type
\begin{equation}\label{Cauchy-Dirichlet-variational}
  \left\{\begin{array}{cl}
       \partial_tu-\dive D_\xi f(x,t,Du)=0 &\mbox{in }\Omega_T,  \\[0.8ex]
       u=g&\mbox{on }\partial_{\mathrm{par}}\Omega_T, 
  \end{array}\right.    
\end{equation} 
for an integrand $f(x,t,\xi):\Omega_T\times\R^n\to\R$ that is convex
and continuously differentiable with
respect to the $\xi$-variable and satisfies
\begin{align}
  \label{fip0}
0\le f(x,t,\xi)&\le b(x,t)(\mu^2+|\xi|^2)^{\frac{q}{2}},\\[0.8ex]
  \label{fip1}
|D_\xi f(x,t,\xi)|&\le b(x,t)(\mu^2+|\xi|^2)^{\frac{q-1}{2}},\\[0.8ex]
\label{fip2}
\langle D_\xi f(x,t,\xi),\xi\rangle&\ge a(x,t)(\mu^2+|\xi|^2)^{\frac {p-2}{2}}|\xi|^2,
\end{align}
for a.e. $(x,t)\in\Omega_T$ and every $\xi\in\R^n$, 
exponents $2\le p\le q$ and parameters  $\mu\in[0,1]$, with functions
$a,b:\Omega_T\to[0,\infty)$ satisfying \eqref{summab} and \eqref{gap}.
For the boundary and initial values, we assume
\begin{equation}
    \label{lateral-condition-f}
    g\in L^{\gamma}(0,T;W^{1,\gamma}(\Omega))
    \qquad\mbox{and}\qquad
    \partial_t g\in L^{p_\alpha'}(0,T;W^{-1,p_\alpha'}(\Omega)),
\end{equation}
where
\begin{equation}
\label{def:gamma}
      \gamma=\frac{\beta p_\alpha}{(\beta-1)p_\alpha -\beta(q-1)},
\end{equation}
and
\begin{equation}
    \label{initial-condition}
    g(\cdot,0)\in L^2(\Omega).
\end{equation}

Next, we specify the notion of variational solution that we rely on for the
existence result.
The concept of variational solutions goes back to an idea by Lichnewsky and Temam \cite{Lichnewsky-Temam:1978}.
In the context of parabolic equations with $p,q$-growth, a related
notion has been used, for instance, in \cite{BDM:pq,BDM:variational}.

\begin{defn}\label{def:var-sol}
  Suppose that $f:\Omega_T\times\R^{n}\to\R$ is an integrand satisfying
  \eqref{fip0}, and the initial and boundary data $g$
  satisfies~\eqref{lateral-condition-f}
  and~\eqref{initial-condition}.
  Then we say that a map
  \begin{equation*}
    u\in \big(g+L^{p_\alpha}(0,T;W^{1,p_\alpha}_0(\Omega))\big)\cap C_w([0,T];L^2(\Omega))
  \end{equation*}
  is a \emph{variational solution} of the
  Cauchy-Dirichlet-problem~\eqref{Cauchy-Dirichlet-variational} if and only if
  $u(\cdot,0)=g(\cdot,0)$ and for all comparison maps $v\in
  g+L^{q_\beta}(0,T;W^{1,q_\beta}_0(\Omega))$
  with $\partial_tv\in
  L^{p_\alpha'}(0,T;W^{-1,p_\alpha'}(\Omega))$ we have
  \begin{align}\label{def:var-inequ}
    \int_{\Omega_\tau}f(x,t,Du)\dx\dt
    &\le
    \int_{\Omega_\tau}f(x,t,Dv)\dx\dt
    +
    \int_0^\tau \langle \partial_tv,v-u\rangle_{W^{-1,p_\alpha'}}\,\dt\cr
    &\qquad
    -\frac12\int_\Omega|(v-u)(\tau)|^2\dx
    +\frac12\int_\Omega|(v-g)(0)|^2\dx,
  \end{align}
  for every $\tau\in(0,T]$, where
  $\langle\cdot,\cdot\rangle_{W^{-1,p_\alpha'}}$ denotes the duality pairing
  between $W^{-1,p_\alpha'}(\Omega)$ and $W^{1,p_\alpha}_0(\Omega)$. 
\end{defn}

Under the above-mentioned assumptions, we prove that there exists a
variational solution to~\eqref{Cauchy-Dirichlet-variational} that is locally
bounded. More precisely, our second main result is the following. 
\begin{thm}\label{thm:exist}
  Let $f:\Omega_T\times\R^n\to\R^n$ be a measurable function
  satisfying~\eqref{fip0} -- \eqref{fip2}, \eqref{summab} and
  \eqref{gap}, for which the partial map $\xi\mapsto
  f(x,t,\xi)$ is convex and continuously differentiable. For the boundary and initial values we
  assume~\eqref{lateral-condition-f} and \eqref{initial-condition}.
  Then there exists a locally bounded
  variational solution $u$ of the Cauchy-Dirichlet
  problem~\eqref{Cauchy-Dirichlet-variational} in the sense of
  Definition~\ref{def:var-sol}. More precisely, for every compact
  subset $K\subset\Omega\times(0,T]$, the solution satisfies
  \begin{equation}\label{local-bound-u}
    \esssup_{K}|u|\le C_K\Big(1+\|\partial_t
    g\|^{p'}_{L^{p_{\alpha}'}-W^{-1,p_{\alpha}'}}
    +\|g\|_{L^{p_\alpha}-W^{1,p_\alpha}}^p
    +\|Dg\|^{\frac{p}{p+1-q}}_{L^\gamma}
   +\mu^{q-1}\|Dg\|_{L^{\beta'}}
    +\|g\|_{L^2-L^\infty}^2\Big)^\vartheta,
  \end{equation}
  with a positive exponent $\vartheta$ 
  that depends on the data $n,p,q,\alpha,\beta$ and a constant $C_K$ that depends on the same
  data and on $\|a^{-1}\|_{L^\alpha}$, $\|b\|_{L^\beta}$, $\mathrm{diam}(\Omega)$ and the parabolic distance $\dist_{\mathrm{par}}^{(p,q)}(K,\partial_{\mathrm{par}}\Omega_T)$
  defined in~\eqref{pq-dist}.

\end{thm}

For the proof, we refer to Section~\ref{sec:exist}. At this point, we only briefly sketch the strategy of the proof. 
For the construction of the asserted  variational solutions, we consider regularized problems with standard $q_\beta$-growth, which have weak solutions in the energy space $L^{q_\beta}(0,T;W^{1,q_\beta}(\Omega))$. This regularity allows us to apply the local bounds stated in Theorem~\ref{thm:subsolution}, which also hold uniformly for the approximating solutions. Moreover, for Dirichlet data satisfying~\eqref{lateral-condition-f}, we derive an energy estimate that ensures that the approximating sequence is bounded in the space $L^{p_\alpha}(0,T;W^{1,p_\alpha}(\Omega))$. At this point, we exploit the variational structure of equation~\eqref{Cauchy-Dirichlet-variational} to pass to the limit. In fact, weak convergence is sufficient to pass to the limit in the variational inequality~\eqref{def:var-inequ} by using the lower semicontinuity of convex functionals. In this way we can verify that the limit of the approximating solutions is indeed a variational solution. Because the approximating solutions are uniformly bounded in the interior of the domain, this procedure yields a variational solution that is locally bounded.

It is worth pointing out that our bound in \eqref{gap} in case $\alpha=\beta=\infty$ doesn't give back that in \cite{singer2}, where  the existence of a bounded variational solution has been proven in the non degenerate context. This is due to the fact that our result relies on a general sup-bound for weak subsolutions instead of a bound specific for the variational ones.

\section{Preliminaries}\label{sec:preliminaries}

\subsection{Notation}

For an open bounded set $\Omega\subset\R^n$ and $T>0$ we consider the space-time
cylinder $\Omega_T:=\Omega\times(0,T)$. The parabolic boundary of $\Omega_T$ is given by 
$$
    \partial_{\mathrm{par}} \Omega_T
    :=
    \big(\overline \Omega\times\{0\}\big) \cup \big(\partial \Omega\times(0,T]\big).
$$

For points in space-time, we will write $z=(x,t)$ or $z_o=(x_o,t_o)$,
where $x,x_o\in\R^n$ denote spatial variables and $t_o,t\in\R$
instants in time. 
For backward parabolic cylinders, we use the notation
\begin{equation*}
  Q_{\rho,\sigma}(z_o):= B_\rho(x_o)\times(t_o-\sigma,t_o)
\end{equation*}
for $\rho,\sigma>0$, where $B_\rho(x_o)\subset\R^n$ denotes the open
ball with radius $\rho>0$ and center $x_o\in\R^n$.
The above cylinders with $\sigma=\rho^{\frac{p}{p+1-q}}$ best reflect
the scaling behaviour of parabolic problems with $p,q$-growth as we consider them here. For these cylinders we use the abbreviation
\begin{equation}\label{pq-cyl}
  Q_\rho^{(p,q)}(z_o)
  :=
  B_\rho(x_o)\times\big(t_o-\rho^{\frac{p}{p+1-q}},t_o\big).
\end{equation}

The associated parabolic distance adapted to $p,q$-growth is given by 
$$
    d_{\mathrm{par}}^{(p,q)}(z_1,z_2)
    :=
    |x_1-x_2| + |t_1-t_2|^{\frac{p+1-q}{p}},
$$
for two points $z_1=(x_1,t_1), z_2=(x_2,t_2)\in\R^{n+1}$.
The corresponding distance of two subsets $K_1,K_2\subset \Omega_T$ is

\begin{equation}
    \dist_\mathrm{par}^{(p,q)}(K_1,K_2)
    :=
    \inf\big\{ d_\mathrm{par}^{(p,q)}(z_1,z_2)\colon z_1\in
    K_1,\,z_2\in K_2\big\}.\label{pq-dist}
\end{equation}

\subsection{A time mollification}

For $v \in L^1(\Omega_T)$ and $h>0$, we define a mollification in time by
\begin{equation}\label{time-molli}
\mollifytime{v}{h}(x,t)
	:= \frac 1 h \int_t^T e^{\frac{t-s}{h}} v(x,s) \mathrm{d}s.
\end{equation}

We state some well-known properties of this mollification in the
following lemma. For the proofs, we refer to~\cite[Appendix B]{BDM:pq}.

\begin{lem} \label{lem:mollifier}
  Let $v$ and $\mollifytime{v}{h}$ be as above. Then the following
  properties hold:
  \begin{enumerate}
  \item If $v \in L^p(\Omega_T)$ for some $p \geq 1$, then
  \begin{equation*}
    \| \mollifytime{v}{h} \|_{L^p(\Omega_T,\R^N)} \leq \| v \|_{L^p(\Omega_T,\R^N)},
  \end{equation*}
  and $\mollifytime{v}{h} \to v$ in $L^p(\Omega_T)$ as $h \to 0$.
\item Let $v\in L^p(0,T; W^{1,p}(\Omega))$ for some
  $p \geq 1$. Then
  \begin{equation*}
    \| \mollifytime{v}{h}\|_{L^p(0,T;W^{1,p}(\Omega))} \leq \| v\|_{L^p(0,T;W^{1,p}(\Omega))} 
  \end{equation*}
  and $\mollifytime{v}{h}\to v$ in $L^p(0,T;W^{1,p}(\Omega))$ as
  $h \to 0$.
\item If $v \in L^p(0,T; W_0^{1,p}(\Omega))$, then
  $\mollifytime{v}{h} \in L^p(0,T;W_0^{1,p}(\Omega))$. 
\item The weak time derivative $\partial_t \mollifytime{v}{h}$
  exists in $\Omega_T$ and is given by 
  \begin{equation*}
    \partial_t \mollifytime{v}{h} = \frac{1}{h} (\mollifytime{v}{h} -v).
  \end{equation*}
\end{enumerate}
\end{lem}

\subsection{An iteration lemma}
  We will rely on the following well-known lemma to absorb
  certain terms into the left-hand side of an estimate. 
  The proof is analogous to the proof of~\cite[ Lemma 6.1]{Giusti}.
   \begin{lem}\label{lem:Giaq}For $0<\rho<\sigma$,
     assume that $f:[\rho,\sigma]\to[0,\infty)$ is a bounded function that satisfies
     \begin{equation*}
       f(\rho')\le\vartheta f(\sigma')+\frac A{(\rho'-\sigma')^\alpha}
                   +\frac B{(\rho'-\sigma')^\beta}+\frac C{(\rho'-\sigma')^\gamma}+D
     \end{equation*}
     for all $\rho<\rho'<\sigma'<\sigma$, where 
      $A,B,C,D\ge0$, $\alpha\ge\beta\ge\gamma\ge 0$
     and $\vartheta\in(0,1)$ are fixed parameters. Then we have the estimate
     \begin{equation*}
       f(\rho)\le c(\alpha,\vartheta)\bigg(\frac A{(\sigma-\rho)^\alpha}+\frac B{(\sigma-\rho)^\beta}+\frac C{(\sigma-\rho)^\gamma}+D\bigg).
     \end{equation*}
   \end{lem}

\subsection{An interpolation type inequality}

\noindent For further needs, we also record the following
		interpolation inequality whose proof can be found in \cite[Chapter~I, Proposition~3.1]{DiBenedettobook}.

\begin{lem}
\noindent \label{lem:inter} Assume that the function $v:Q_{r,s}(z_{0})\rightarrow\mathbb{R}$ satisfies
$$v\in
L^{\infty}\big(t_{0}-s,t_{0};L^{2}\left(B_{r}\left(x_{0}\right)\right)\big)\cap
L^{p_\alpha}\big(t_{0}-s,t_{0};W^{1,p_\alpha}_0\left(B_{r}\left(x_{0}\right)\big)\right)$$
for an exponent $p_\alpha:=\frac{p\alpha}{\alpha+1}\ge 1$.
Then the following estimate
\[
  \mint_{Q_{r,s}(z_{0})}\left|v\right|^{\frac{p_\alpha(n+2)}{n}}\,\dz\leq c\left(\sup_{t\in(t_{0}-s,t_{0})}\mint_{B_{r}(x_{0})}\left|v(x,t)\right|^{2}\,\dx\right)^{\frac{p_\alpha}{n}}r^{p_\alpha}\mint_{Q_{r,s}(z_{0})}\left|Dv\right|^{p_\alpha}\,\dz
\]
holds true for a positive constant $c$ depending at most on $n$,
$p$ and $\alpha$.
\end{lem}

\subsection{Lemma on geometric convergence}
Finally, we recall the well-known lemma on fast geometric convergence, see e.g. \cite[Lemma 7.1]{Giusti}.
\begin{lem}\label{lem:geom-conv}
Let $(X_i)_{i\in \N_0}$ be a sequence of positive real 
numbers satisfying the recursive inequalities
\[
    X_{i+1}\le C\lambda^i X_i^{1+\kappa}
    \quad\mbox{for all }i\in\N_0,
\]
where $C,\kappa>0$ and $\lambda>1$ are given numbers. If 
\[
    X_0\le C^{-1/\kappa} \lambda^{-1/\kappa^2}
\]
then $X_i\to 0$ in the limit $i\to\infty$.
\end{lem}

\section{Local boundedness from above of weak subsolutions}  \label{sec:subsolution}
This section is devoted to the proof of Theorem~\ref{thm:subsolution}. For
later applications in the proof of the existence result, we impose assumptions
that are even more general than in
Theorem~\ref{thm:subsolution} and include regularizations of equation~\eqref{equa0}.
To this end, for a parameter $\varepsilon\in[0,1]$, we assume that the
Carath\'eodory function $\A:\Omega_T\times\R^{n}\to\R^{n}$ satisfies 
\begin{equation}\label{ip1-eps}
|\mathcal{A}(x,t,\xi)|\le b(x,t)(\mu^2+|\xi|^2)^{\frac{q-1}{2}}
+
\varepsilon |\xi|^{q_\beta-1},
\end{equation}
\begin{equation}\label{ip2-eps}
\langle \A(x,t,\xi),\xi\rangle\ge a(x,t)(\mu^2+|\xi|^2)^{\frac {p-2}{2}}|\xi|^2+\varepsilon |\xi|^{q_\beta},
\end{equation}
for almost every $(x,t)\in\Omega_T$ and all $\xi\in\real^{n}$, with $2\le p\le q$, $\mu\in[0,1]$ and functions
$a,b:\Omega_T\to[0,\infty)$ satisfying \eqref{summab} and \eqref{gap}.

\subsection{A Caccioppoli type inequality}
The first step in the proof of the local sup-bound is the following
Caccioppoli type inequality. 

   \begin{lem}\label{lem:caccioppoli}
    Let $u\in L^{q_\beta}(0,T; W^{1,q_\beta}(\Omega))$ be a weak subsolution to~\eqref{equa0}, under assumptions~\eqref{ip1-eps}, \eqref{ip2-eps}, \eqref{summab} and \eqref{gap}. Then, for any pair of concentric cylinders
    $Q_{\rho,\sigma}(z_o)\Subset Q_{r,s}(z_o)$ with $\dist(Q_{r,s}(z_o),\partial_{\mathrm{par}}\Omega_T)>0$ and any
    constant $k\in\R$, we have the  estimate
    \begin{align*}
      &\sup_{t\in (t_o-\sigma,t_o)}\int_{B_\rho(x_o)\times\{t\}}(u-k)_+^2\d x
      +
        \big\|a^{-1}\big\|_{L^\alpha}^{-1}\,\bigg(\int_{Q_{\rho,\sigma}(z_o)}
        |D(u-k)_+|^{p_\alpha}\d z\bigg)^{\frac{\alpha+1}{\alpha}}\\
      &\qquad\qquad+
      \varepsilon\int_{Q_{\rho,\sigma}(z_o)}|D(u-k)_+|^{q_\beta}\,\dz\\
      &\qquad\le
      \frac{c\mu^{q-1}}{r-\rho}\|b\|_{L^\beta}
    \left( \int_{Q_{r,s}(z_o)}(u-k)_+^{\beta'}\,\dz\right)^\frac{1}{\beta'}
     +
     \Bigg[\frac{c}{r-\rho}\|b\|_{L^\beta}\big\|a^{-1}\big\|_{L^\alpha}^{\frac{q-1}{p}}
     \left(\int_{Q_{r,s}(z_o)}(u-k)_+^{\gamma}\,\dz\right)^{\frac{1}{\gamma}}\Bigg]^{\frac{p}{p+1-q}} \cr
     &\qquad\quad+\frac{c \varepsilon}{(r-\rho)^{q_\beta}}
    \int_{Q_{r,s}(z_o)}(u-k)_+^{q_\beta}\,\dz +
     \frac{c}{s-\sigma}\int_{Q_{r,s}(z_o)} (u-k)_+^2\,\dz.
    \end{align*}
    The constant $c$ in the above estimate depends only on
    $n,p,q,\alpha$ and $\beta$ but is independent of $\varepsilon$,
    and the exponent $\gamma$ in the second integral on the right-hand
    side is defined in~\eqref{def:gamma}.
  \end{lem}
\begin{rem}\label{rem:gap}
 When comparing the exponents of $(u-k)_+$ on the right-hand side, we note that $\gamma\ge\beta'$ holds by definition, and the facts $q\ge1$ and $p_\alpha\le p\le q$ imply 
\begin{equation}\label{bounds-gamma}
  \gamma\ge\frac{\gamma}{\beta'}\ge\frac{p_\alpha}{p_\alpha+1-q}\ge\frac{p}{p+1-q}\ge q\ge2.
\end{equation}

Therefore, assumption \eqref{gap} implies 
\begin{equation}\label{upper-bound-gamma}
  \max\{\gamma,2,\beta',q_\beta\}=\gamma<\frac{p_\alpha(n+2)}{n}.
\end{equation}
\end{rem}
  
  \begin{proof}
  We omit the center $z_o=(x_o,t_o)$ from the notation for convenience.
  For fixed parameters $k\in\R$ and $h>0$, we set $u_{k,h}=(\mollifytime{u}{h}-k)_+$, with the time mollification $\mollifytime{u}{h}$ defined in \eqref{time-molli}. In Definition \ref{def:weak-solution}, we choose the test function $\varphi=\Phi u_{k,h}$, for a non-negative function $\Phi\in C^\infty_0(Q_{r,s})$ that will be specified later. This yields
  \begin{align*}
   0 &\le\int_{\Omega_T}u\partial_t(\Phi u_{k,h})\,\dz
     -
     \int_{\Omega_T}\Phi\langle \A(x,t,Du), Du_{k,h}\rangle\,\dz 
     -
     \int_{\Omega_T}\langle \A(x,t,Du),D\Phi \otimes u_{k,h}\rangle\,\dz\cr
     &= \int_{\Omega_T\cap \{\mollifytime{u}{h}>k\}} u\partial_t(\Phi (\mollifytime{u}{h}-k))\, \dz-\int_{\Omega_T\cap \{\mollifytime{u}{h}>k\}}\Phi\langle \A(x,t,Du), D\mollifytime{u}{h}\rangle\,\dz\cr
     &\qquad-\int_{\Omega_T\cap \{\mollifytime{u}{h}>k\}}\langle \A(x,t,Du), D\Phi \otimes (\mollifytime{u}{h}-k)\rangle\,\dz\cr
     &=:I_1-I_2-I_3,
  \end{align*}
  which is equivalent to
  \begin{equation}\label{equation-I-123}
      -I_1+I_2\le-I_3.
  \end{equation}
  We rewrite the first term as 
  \begin{align*}
    I_1
    &= 
   \int_{\Omega_T\cap \{\mollifytime{u}{h}>k\}} u\partial_t(\mollifytime{u}{h}-k)\Phi\,\dz
   +
   \int_{\Omega_T\cap \{\mollifytime{u}{h}>k\}} u(\mollifytime{u}{h}-k)\partial_t\Phi \,\dz\\
   &=
   \int_{\Omega_T\cap \{\mollifytime{u}{h}>k\}} (\mollifytime{u}{h}-k)\partial_t(\mollifytime{u}{h}-k)\Phi\,\dz
   +
   k\int_{\Omega_T\cap \{\mollifytime{u}{h}>k\}} \partial_t(\mollifytime{u}{h}-k)\Phi\,\dz\\
   &\qquad+\int_{\Omega_T\cap \{\mollifytime{u}{h}>k\}} (u-\mollifytime{u}{h})\partial_t(\mollifytime{u}{h}-k)\Phi\,\dz
   +\int_{\Omega_T\cap \{\mollifytime{u}{h}>k\}} u(\mollifytime{u}{h}-k)\partial_t\Phi \,\dz.
  \end{align*}
  We integrate by parts in the first two integrals on the right-hand side. In the third integral, we use Lemma \ref{lem:mollifier}\,(iv), which implies $(u-\mollifytime{u}{h})\partial_t(\mollifytime{u}{h}-k)\le0$. In this way, we get 
  \begin{align*}
  I_1
  &\le -\frac12 \int_{\Omega_T} (\mollifytime{u}{h}-k)^2_+\partial_t\Phi\,\dz
  +
  \int_{\Omega_T} (u-k)(\mollifytime{u}{h}-k)_+\partial_t\Phi \,\dz.
  \end{align*}
  Letting $h\downarrow0$ and using Lemma~\ref{lem:mollifier}\,(i), we deduce 
  \begin{align}\label{est-I1}
    \limsup_{h\downarrow0}I_1\le \frac12 \int_{\Omega_T} (u-k)^2_+\partial_t\Phi\,\dz.
  \end{align}
  Next, we observe that assumption~\eqref{ip1-eps} implies 
  \begin{equation*}
      |\mathcal{A}(x,t,Du)|
      \le
      b(x,t)(\mu^2+|Du|^2)^{\frac{q-1}{2}}+\varepsilon |Du|^{q_\beta-1}.
  \end{equation*}
  In view of the assumptions $b\in L^\beta(\Omega_T)$ and $Du\in L^{q_\beta}(\Omega_T)$, this yields 
  \begin{equation}\label{A-in-dual-space}
    |\mathcal{A}(x,t,Du)|\in L^{\frac{\beta q}{\beta(q-1)+1}}(\Omega_T)=\big[L^{q_\beta}(\Omega_T)\big]'.
\end{equation}    
  Because Lemma~\ref{lem:mollifier}\,(ii) yields the convergence $D\mollifytime{u}{h}\to Du$ in $L^{q_\beta}(\Omega_T,\R^n)$, as $h\downarrow0$, we obtain 
  \begin{align*}
    \lim_{h\downarrow0}I_2
    =
    \lim_{h\downarrow0}
    \int_{\Omega_T}\Phi\langle \A(x,t,Du),D(\mollifytime{u}{h}-k)_+\rangle\,\dz
    =
    \int_{\Omega_T}\Phi\langle \A(x,t,Du),D(u-k)_+\rangle\,\dz.
  \end{align*}
  
  We estimate the right-hand side by means of assumption \eqref{ip2-eps} and obtain
  \begin{eqnarray}\label{est-I2}
     \lim_{h\downarrow0}I_2&=& \int_{\Omega_T\cap \{u>k\}} \Phi\langle \A(x,t,Du),Du\rangle\,\dz\cr\cr
     &\ge& \int_{\Omega_T\cap \{u>k\}} a(x,t) |Du|^p\Phi\,\dz+ \varepsilon \int_{\Omega_T\cap \{u>k\}}|Du|^{q_\beta}\Phi \,\dz. 
  \end{eqnarray}
  Using again~\eqref{A-in-dual-space} and the convergence properties of the mollification from Lemma~\ref{lem:mollifier}, we deduce
  \begin{equation}
      \lim_{h\downarrow0}I_3
      =
      \int_{\Omega_T}\langle \A(x,t,Du),D\Phi \otimes (u-k)_+\rangle\,\dz.
  \end{equation}
  Therefore, assumption \eqref{ip1-eps} gives
  \begin{equation}\label{est-I3}
  \lim_{h\downarrow0}|I_3|
  \le 
  \int_{\Omega_T} b(x,t)(\mu^2+|Du|^2)^{\frac{q-1}{2}}| D\Phi |(u-k)_+\,\dz+ \varepsilon \int_{\Omega_T}|Du|^{q_\beta-1}| D\Phi |(u-k)_+\,\dz.
  \end{equation}
  Letting $h\downarrow0$ in \eqref{equation-I-123} and using \eqref{est-I1}, \eqref{est-I2} and \eqref{est-I3}, we arrive at the bound
   \begin{align}\label{est2}
   &-\frac12 \int_{\Omega_T} (u-k)_+^2\partial_t \Phi \,\dz+ \int_{\Omega_T} a(x,t)|D(u-k)_+|^p \Phi\,\dz+ \varepsilon \int_{\Omega_T}|D(u-k)_+|^{q_\beta}\Phi \,\dz\nonumber\\
   &\qquad\le \int_{\Omega_T} b(x,t)(\mu^2+|Du|^2)^{\frac{q-1}{2}}| D\Phi |(u-k)_+\,\dz\nonumber\\
   &\qquad+\varepsilon \int_{\Omega_T}|Du|^{q_\beta-1}| D\Phi |(u-k)_+\,\dz. 
  \end{align}
  For a fixed time $t_1\in(t_o-\sigma,t_o)$ and $\delta\in(0,t_o-t_1)$, we define the Lipschitz continuous function $\bar\chi:(0,T)\to\R$ by
  \begin{equation}
    \bar\chi(t)=
    \left\{\begin{array}{l l}1 &\mathrm{if }\quad t\le t_1,\\
    \mathrm{affine } & \mathrm{if }\quad t_1<t< t_1+\delta,\\
     0 & \mathrm{if }\quad t\ge t_1+\delta.
    \end{array}\right.
  \end{equation}
  We consider a second cut-off function in time $\chi\in C^{0,1}((0,T),[0,1])$ which is chosen to be piecewise affine with
  $\chi(t)\equiv 0$ on $(0,t_o-s)$, $\chi(t)\equiv 1$ on
$(t_o-\sigma, T)$ and $\partial_t \chi\equiv \frac{1}{s-\sigma}$
 on $(t_0-s,t_0-\sigma)$. 
  Moreover, we choose 
  a cut-off function in space $\varphi\in C^\infty_0(B_r,[0,1])$ satisfying $\varphi\equiv 1$ in $B_{\rho}$ and $|D\varphi|\le \frac{2}{r-\rho}$ on $B_r$.
  Then we choose $\Phi(x,t)=\chi(t)\bar\chi(t)\varphi(x)$ in \eqref{est2}. 
  By letting $\delta\to 0$ we obtain
  \begin{align*}
     &\frac12\int_{\Omega\times\{t_1\}}(u-k)_+^2\,\chi\varphi\,\dx+ \int_{\Omega_{t_1}} a(x,t)|D(u-k)_+|^p \chi\varphi\,\dz+ \varepsilon \int_{\Omega_{t_1}}|D(u-k)_+|^{q_\beta}\chi\varphi  \,\dz\nonumber\\
     &\qquad\le 
     \int_{Q_{r,s}} b(x,t)(\mu^2+|Du|^2)^{\frac{q-1}{2}}\chi|D\varphi|(u-k)_+\,\dz +\varepsilon\int_{Q_{r,s}} |Du|^{q_\beta-1}\chi|D\varphi|(u-k)_+\,\dz
     \nonumber\\
     &\qquad+
     \frac12 \int_{Q_{r,s}} (u-k)_+^2\partial_t\chi\,\varphi \,\dz
  \end{align*}
  for almost every $t_1\in (t_o-\sigma,t_o)$. 
  We use this inequality in two ways. First, we omit the second and third integral on the left-hand side and take the supremum over $t_1\in(t_o-\sigma,t_o)$. Second, we omit the first integral and let $t_1\to t_o$. In this way, we deduce
 \begin{align}\label{pre-caccio}
     &\frac12\sup_{t\in(t_o-\sigma,t_o)}\int_{B_\rho\times\{t\}}(u-k)_+^2 \,\dx
     + 
     \int_{Q_{\rho,s}} a(x,t)|D(u-k)_+|^p \chi\,\dz+ \varepsilon \int_{Q_{\rho,s}}|D(u-k)_+|^{q_\beta}\chi  \,\dz\cr
     &\qquad\le 
     2\int_{Q_{r,s}} b(x,t)(\mu^2+|Du|^2)^{\frac{q-1}{2}}\chi|D\varphi|(u-k)_+\,\dz 
     +2\varepsilon\int_{Q_{r,s}} |Du|^{q_\beta-1}\chi|D\varphi|(u-k)_+\,\dz\nonumber\\
     &\qquad\qquad+
     \int_{Q_{r,s}} (u-k)_+^2\partial_t\chi\,\dz\cr
     &\qquad\le 
     2^{q-1}\mu^{q-1}\int_{Q_{r,s}} b(x,t)\chi|D\varphi|(u-k)_+\,\dz
     +
     2^{q-1}\int_{Q_{r,s}} b(x,t)|Du|^{q-1}\chi|D\varphi|(u-k)_+\,\dz \cr
&\qquad\qquad+
     2\varepsilon\int_{Q_{r,s}} |Du|^{q_\beta-1}\chi|D\varphi|(u-k)_+\,\dz+
     \int_{Q_{r,s}} (u-k)_+^2\partial_t\chi\,\dz. 
  \end{align}  
We estimate the first integral in the right-hand side with the use of H\"older's inequality, obtaining
\begin{align}\label{equa2b}
  2^{q-1}\mu^{q-1}\int_{Q_{r,s}} b(x,t)\chi|D\varphi|(u-k)_+\,\dz 
  \le 2^{q-1}\mu^{q-1}\|b\|_{L^\beta}
    \left( \int_{Q_{r,s}}(u-k)_+^{\beta'}|D\varphi|^{\beta'}\,\dz\right)^\frac{1}{\beta'}.
\end{align}
For the estimate of the second integral on the right-hand side of~\eqref{pre-caccio}, we use H\"older's inequality with exponents
\begin{equation}\label{choice-gamma} 
  \beta,\,\, \frac{p_\alpha}{q-1} \qquad \text{and}\quad \gamma=\frac{\beta p_\alpha}{(\beta-1)p_\alpha-\beta(q-1)}\,.
\end{equation}
Therefore, we obtain
\begin{align}\label{equa3}
& 2^{q-1}\int_{Q_{r,s}}b(x,t)|Du|^{q-1}\chi|D\varphi|(u-k)_+\,\dz  \cr
&\quad\le2^{q-1}\|b\|_{L^\beta}
\left(\int_{Q_{r,s}}|D(u-k)_+|^{p_\alpha} \chi^{\frac{p_\alpha}{q-1}}\,\dz\right)^{\frac{q-1}{p_\alpha}}
\left(\int_{Q_{r,s}}(u-k)_+^{\gamma}|D\varphi|^{\gamma}\,\dz\right)^{\frac{1}{\gamma}}.
\end{align}
For the estimate of the third one, we use H\"older's inequality with exponents $\frac{q_\beta}{q_\beta-1}$ and $q_\beta$   
and  
we deduce that
\begin{align}\label{equa4.14}
2\varepsilon &\int_{Q_{r,s}} |Du|^{q_\beta-1}\chi|D\varphi|(u-k)_+\,\dz\cr
&\quad\le2\varepsilon
\left(\int_{Q_{r,s}}|D(u-k)_+|^{q_\beta} \chi\,\dz\right)^{\frac{q_\beta-1}{q_\beta}}
\left(\int_{Q_{r,s}}(u-k)_+^{q_\beta}\chi|D\varphi|^{q_\beta}\,\dz\right)^{\frac{1}{q_\beta}}.
\end{align}

On the other hand, the second integral in the left-hand side of~\eqref{pre-caccio} can be estimated from below by H\"older's inequality with exponents $\alpha+1$ and $\frac{\alpha+1}{\alpha}$ as 
\begin{align}\label{equa4}
    \int_{Q_{\rho,s}}a(x,t)|D(u-k)_+|^p\chi\,\dz 
    &\ge
    \big\|a^{-1}\big\|_{L^\alpha}^{-1}\left(\int_{Q_{\rho,s}}a^{-\frac{\alpha}{\alpha+1}}\big[a\,|D(u-k)_+|^{p}\chi\big]^{\frac{\alpha}{\alpha+1}}\,\dz\right)^{\frac{\alpha+1}{\alpha}}\cr
    &\ge
    \big\|a^{-1}\big\|_{L^\alpha}^{-1}\left(\int_{Q_{\rho,s}}|D(u-k)_+|^{p_\alpha}\chi^{\frac{p_\alpha}{q-1}}\,\dz\right)^{\frac{\alpha+1}{\alpha}}.
\end{align}
In the last step, we also used the fact $p>q-1$, which implies $\chi^{\frac{\alpha}{\alpha+1}}\ge\chi^{\frac{p_\alpha}{q-1}}$.
Combining \eqref{pre-caccio} with \eqref{equa2b}, \eqref{equa3}, \eqref{equa4.14} and \eqref{equa4}, we conclude
 \begin{align}\label{pre-caccio-2}
     &\frac12\sup_{t\in(t_o-\sigma,t_o)}\int_{B_\rho\times\{t\}}(u-k)_+^2 \,\dx
     + 
     \big\|a^{-1}\big\|_{L^\alpha}^{-1}\left(\int_{Q_{\rho,s}}|D(u-k)_+|^{p_\alpha}\chi^{\frac{p_\alpha}{q-1}}\,\dz\right)^{\frac{\alpha+1}{\alpha}}\cr&\qquad\qquad+ \varepsilon \int_{Q_{\rho,s}}|D(u-k)_+|^{q_\beta}\chi  \,\dz\cr
     &\qquad\le 
     2^{q-1}\mu^{q-1}\|b\|_{L^\beta}
    \left( \int_{Q_{r,s}}(u-k)_+^{\beta'}|D\varphi|^{\beta'}\,\dz\right)^\frac{1}{\tilde\beta'}\cr
     &\qquad\qquad+
     2^{q-1}\|b\|_{L^\beta}
     \left(\int_{Q_{r,s}}|D(u-k)_+|^{p_\alpha}\chi^{\frac{p_\alpha}{q-1}} \,\dz\right)^{\frac{q-1}{p_\alpha}}
     \left(\int_{Q_{r,s}}(u-k)_+^{\gamma}|D\varphi|^{\gamma}\,\dz\right)^{\frac{1}{\gamma}}\cr
      &\qquad\qquad+
    2\varepsilon
\left(\int_{Q_{r,s}}|D(u-k)_+|^{q_\beta} \chi\,\dz\right)^{\frac{q_\beta-1}{q_\beta}}
\left(\int_{Q_{r,s}}(u-k)_+^{q_\beta}\chi|D\varphi|^{q_\beta}\,\dz\right)^{\frac{1}{q_\beta}} \cr
     &\qquad\qquad+
     \int_{Q_{r,s}} (u-k)_+^2\partial_t\chi \,\dz. 
  \end{align}  
 Next, we apply Young's inequality with exponents $\frac{p}{q-1}$ and $\frac{p}{p+1-q}$ in the second term on the right-hand side and with exponents $\frac{q_\beta}{q_\beta-1}$ and $q_\beta$ in the third one. Moreover, we use the bounds $|D\varphi|\le\frac{2}{r-\rho}$ and $0\le\partial_t\chi\le\frac{1}{s-\sigma}$. This leads to the estimate 
 \begin{align}\label{pre-caccio-3}
     &\frac12\sup_{t\in(t_o-\sigma,t_o)}\int_{B_\rho\times\{t\}}(u-k)_+^2\,\dx
     + 
     \big\|a^{-1}\big\|_{L^\alpha}^{-1}\left(\int_{Q_{\rho,s}}|D(u-k)_+|^{p_\alpha}\chi^{\frac{p_\alpha}{q-1}}\,\dz\right)^{\frac{\alpha+1}{\alpha}}\cr&\qquad\qquad+ \varepsilon \int_{Q_{\rho,s}}|D(u-k)_+|^{q_\beta}\chi  \,\dz\cr
     &\qquad\le 
     \frac{c\mu^{q-1}}{r-\rho}\|b\|_{L^\beta}
    \left( \int_{Q_{r,s}}(u-k)_+^{\beta'}\,\dz\right)^\frac{1}{\beta'}\cr
     &\qquad\qquad+
     \frac{c}{r-\rho}\|b\|_{L^\beta}
     \frac{\big\|a^{-1}\big\|_{L^\alpha}^{\frac{q-1}{p}}}{\big\|a^{-1}\big\|_{L^\alpha}^{\frac{q-1}{p}}}
     \left(\int_{Q_{r,s}}|D(u-k)_+|^{p_\alpha}\chi^{\frac{p_\alpha}{q-1}} \,\dz\right)^{\frac{q-1}{p_\alpha}}
     \left(\int_{Q_{r,s}}(u-k)_+^{\gamma}\,\dz\right)^{\frac{1}{\gamma}} \cr
     &\qquad\qquad+
     \frac{c}{r-\rho}\varepsilon
\left(\int_{Q_{r,s}}|D(u-k)_+|^{q_\beta} \chi\,\dz\right)^{\frac{q_\beta-1}{q_\beta}}
\left(\int_{Q_{r,s}}(u-k)_+^{q_\beta}\,\dz\right)^{\frac{1}{q_\beta}} \cr
     &\qquad\qquad+
     \frac{1}{s-\sigma}\int_{Q_{r,s}} (u-k)_+^2\,\dz\cr
     &\qquad\le
     \frac12\big\|a^{-1}\big\|_{L^\alpha}^{-1}\left(\int_{Q_{r,s}}|D(u-k)_+|^{p_\alpha}\chi^{\frac{p_\alpha}{q-1}}\,\dz\right)^{\frac{\alpha+1}{\alpha}}+\frac{\varepsilon}{2}
\int_{Q_{r,s}}|D(u-k)_+|^{q_\beta} \chi\,\dz\cr
     &\qquad\qquad+
     \frac{c\mu^{q-1}}{r-\rho}\|b\|_{L^\beta}
    \left( \int_{Q_{r,s}}(u-k)_+^{\beta'}\,\dz\right)^\frac{1}{\beta'}\cr
     &\quad\quad+
     \Bigg[\frac{c}{r-\rho}\|b\|_{L^\beta}\big\|a^{-1}\big\|_{L^\alpha}^{\frac{q-1}{p}}
     \left(\int_{Q_{r,s}}(u-k)_+^{\gamma}\,\dz\right)^{\frac{1}{\gamma}}\Bigg]^{\frac{p}{p+1-q}} \cr
     &\qquad\quad+\frac{c \varepsilon}{(r-\rho)^{q_\beta}}
    \int_{Q_{r,s}}(u-k)_+^{q_\beta}\,\dz +
     \frac{1}{s-\sigma}\int_{Q_{r,s}} (u-k)_+^2\,\dz, 
  \end{align}
  with a constant $c=c(n,p,q,\alpha,\beta)$ that is independent of $\varepsilon$.    
  In this estimate, we may replace the radii $\rho,r$ by arbitrary radii $\rho',r'$ with $\rho<\rho'<r'<r$. Therefore, Lemma~\ref{lem:Giaq} is applicable, which allows to discard of the first and of the second integrals on the right-hand side. This yields the asserted estimate and concludes the proof of the lemma.
\end{proof}

\subsection{A local sup-estimate for weak subsolutions}

Now we are in a position to prove the local boundedness from above of weak subsolutions of~\eqref{equa0} under assumptions \eqref{ip1-eps} and \eqref{ip2-eps} for small values of $\eps\ge0$. 

\begin{prop}\label{prop:apriori}
   Assume that $u\in L^{q_\beta}(0,T; W^{1,q_\beta}(\Omega))$ is a weak subsolution of~\eqref{equa0} in the sense of Definition~\ref{def:weak-solution}, under assumptions~\eqref{ip1-eps}, \eqref{ip2-eps}, \eqref{summab} and \eqref{gap}.
   Then, for every cylinder $Q:=Q_{2\rho,2\sigma}(z_o)$ with
   $\dist(Q,\partial_{\mathrm{par}}\Omega_T)>0$ for parameters $\rho,\sigma>0$, we have the following estimate on the smaller cylinder $\frac12 Q:=Q_{\rho,\sigma}(z_o)$.
    \begin{align*}
        &\esssup_{\frac12 Q}u\\
        &\quad\le
        \mathrm{C}_{a,b}\max\left\{\bigg(\frac{\sigma}{\rho^{\frac{p}{p+1-q}}}\bigg)^{\vartheta_1}\frac{1}{\rho^{(q-p)\vartheta_2}}\bigg(\mint_{Q} u_+^m\,\d z\bigg)^{\vartheta_3},\bigg(\mint_{Q} u_+^m\,\d z\bigg)^{\frac1m},\bigg(\frac{\rho^{\frac{p}{p+1-q}}}{\sigma}\bigg)^{\frac{p+1-q}{p-2+2(q-p)}},\rho\right\}
        =:k,
    \end{align*} 
    provided
    \begin{equation}\label{assumption-eps}
     0\le\varepsilon\le\widetilde{\mathrm{C}}_{a,b}\left(\frac{k}{\rho}\right)^{\frac{p}{p+1-q}-q_\beta}.
  \end{equation} 
    Here, we abbreviated
    $m:=\frac{p_\alpha(n+2)}{n}$. The constants $\vartheta_i>0$, $i\in\{1,2,3\}$, depend at most on $n,p,q,\alpha,$ and $\beta$ and the constants $\mathrm{C}_{a,b}$ and $\widetilde{\mathrm{C}}_{a,b}$ depend additionally on the quantities
    \begin{equation*}
        \mint_{Q}a^{-\alpha}\,\dx\dt
        \qquad\mbox{and}\qquad
         \mint_{Q}b^{\beta}\,\dx\dt.
    \end{equation*}
    The exact form of dependency on these quantities can be retrieved from~\eqref{choice-k} and \eqref{condition-eps}.
\end{prop}

\begin{proof}
  To simplify the notation, we assume $z_o=(x_o,t_o)=(0,0)$. For given parameters $\rho,\sigma>0$ with $\dist(Q_{2\rho,2\sigma},\partial_{\mathrm{par}}\Omega_T)>0$, we let
  \begin{align*}
    \rho_i&:= \rho+\tfrac1{2^i}\rho,\qquad
    \tilde\rho_i:= \tfrac12(\rho_i+\rho_{i+1}),\\
    \sigma_i&:= \sigma+\tfrac1{2^i}\sigma,\qquad
    \tilde\sigma_i:= \tfrac12(\sigma_i+\sigma_{i+1}),
  \end{align*}
  for $i\in\N_0$. Moreover, we use the abbreviations
  \begin{equation*}
    Q_i:=Q_{\rho_i,\sigma_i}, \qquad \widetilde Q_i:=Q_{\tilde\rho_i,\tilde\sigma_i},\qquad
    B_i:=B_{\rho_i} \quad \mbox{and}\quad\widetilde B_i:=B_{\tilde\rho_i}.
\end{equation*}
  With this definition, we have $Q_{i+1} \subset \widetilde Q_i\subset
  Q_i\subset Q_0=Q_{2\rho,2\sigma}$ and $\dist(Q_i,\partial_{\mathrm{par}}\Omega_T)>0$ for all $i\in\N_0$.
  We write
  \begin{equation*}
      \AA:= \bigg(\mint_{Q_0}a^{-\alpha}\,\dx\dt\bigg)^{\frac{1}{\alpha}}
         =\frac{c\big\|a^{-1}\big\|_{L^\alpha}}{\big(\rho^n\sigma\big)^{\frac1\alpha}}
      \qquad\mbox{and}\qquad
      \BB:= \bigg(\mint_{Q_0}b^{\beta}\,\dx\dt\bigg)^{\frac{1}{\beta}}
      =\frac{c\|b\|_{L^\beta}}{(\rho^n\sigma)^{\frac1\beta}}
  \end{equation*}
  for the norms of $a^{-1}$ and $b$ taken with respect to mean integrals. 
  Then we choose the level $k$ according to 
  \begin{align}\label{choice-k}
      k:= \max\Bigg\{c&\bigg[\frac{\sigma}{\AA}\bigg(\frac{\AA\BB}{\rho}\bigg)^{\frac{p}{p+1-q}}\bigg]^{\vartheta_1}
      \bigg(\frac{\AA\BB}{\rho^{q-p}}\bigg)^{\vartheta_2}
      \bigg(\mint_{Q_0}u_+^m\,\d z\bigg)^{\vartheta_3},\cr
      &\bigg(\mint_{Q_0} u_+^m\,\d z\bigg)^{\frac1m},\bigg[\frac{\AA}{\sigma}\bigg(\frac{\rho}{\AA\BB}\bigg)^{\frac{p}{p+1-q}}\bigg]^{\frac{p+1-q}{p-2+2(q-p)}},\frac{\rho\mu^{p+1-q}}{\AA\BB},\frac{\rho}{(\AA\BB)^{\frac{1}{q-p}}}\Bigg\},
  \end{align}
  with constants $c,\vartheta_1,\vartheta_2,\vartheta_3$ that will be chosen later in dependence on $n,p,q,\alpha$ and $\beta$.
  For this value of $k$ we define
  \begin{equation*}
    k_i:= k-\tfrac1{2^i}k
    \qquad\mbox{for $i\in\N_0$.}
  \end{equation*}
  The Caccioppoli inequality from Lemma~\ref{lem:caccioppoli}
  on the cylinders $\widetilde Q_i\subset Q_i$ with the level
  $k_{i+1}$ gives the bound 
  \begin{align*}
      &\sup_{t\in (-\tilde\sigma_i,0)}\int_{\widetilde
        B_i\times\{t\}}\big(u-k_{i+1}\big)_+^2\d x
      +
      \big\|a^{-1}\big\|_{L^\alpha(Q_0)}^{-1}\,\bigg(\int_{\widetilde
        Q_i}\big|D\big(u-k_{i+1}\big)_+\big|^{p_\alpha}\d x\d t\bigg)^{\frac{\alpha+1}{\alpha}}\\
      &\qquad\le
      \frac{c\mu^{q-1}2^i}{\rho}\|b\|_{L^\beta(Q_0)}
    \left( \int_{Q_i}\big(u-k_{i+1}\big)_+^{\beta'}\,\dz\right)^\frac{1}{\beta'}\\
     &\qquad\qquad+
     \Bigg[\frac{c2^i}{\rho}\|b\|_{L^\beta(Q_0)}\big\|a^{-1}\big\|_{L^\alpha(Q_0)}^{\frac{q-1}{p}}
     \left(\int_{Q_i}\big(u-k_{i+1}\big)_+^{\gamma}\,\dz\right)^{\frac{1}{\gamma}}\Bigg]^{\frac{p}{p+1-q}} \cr
     &\qquad\qquad
     +\frac{c 2^{iq_\beta}\varepsilon}{\rho^{q_\beta}}
    \int_{Q_{i}}\big(u-k_{i+1}\big)_+^{q_\beta}\,\dz  +
     \frac{c2^{i}}{\sigma}\int_{Q_i} \big(u-k_{i+1}\big)_+^2\,\dz,
  \end{align*}
  for every $i\in\N_0$.
  We take means by dividing both sides of the estimate by $|Q_0|=c(n)\rho^n\sigma$. Recalling the facts $\rho_i,\tilde\rho_i\in[\frac\rho2,\rho]$ and $\sigma_i,\tilde\sigma_i\in[\frac\sigma2,\sigma]$ and the definition of $\gamma$  according to~\eqref{choice-gamma}, we deduce  
   \begin{align}\label{caccio-means}
      &\sup_{t\in (-\tilde\sigma_i,0)}\mint_{\widetilde
        B_i\times\{t\}}\frac{\big(u-k_{i+1}\big)_+^2}{\sigma}\,\d x
      +
      \AA^{-1}\bigg(\mint_{\widetilde
        Q_i}\big|D\big(u-k_{i+1}\big)_+\big|^{p_\alpha}\d x\d t\bigg)^{\frac{\alpha+1}{\alpha}}\cr
      &\qquad\le
      \frac{c\mu^{q-1}2^i}{\rho}\BB
    \left( \mint_{Q_i}\big(u-k_{i+1}\big)_+^{\beta'}\d z\right)^\frac{1}{\beta'}
     +
     \Bigg[\frac{c2^i}{\rho}\AA^{\frac{q-1}{p}}\BB
     \left(\mint_{Q_i}\big(u-k_{i+1}\big)_+^{\gamma}\d z\right)^{\frac{1}{\gamma}}\Bigg]^{\frac{p}{p+1-q}}  \cr
     &\qquad\qquad 
     +\frac{c 2^{iq_\beta}\varepsilon}{\rho^{q_\beta}}
    \mint_{Q_{i}}\big(u-k_{i+1}\big)_+^{q_\beta}\,\dz  +
     \frac{c2^{i}}{\sigma}\mint_{Q_i} \big(u-k_{i+1}\big)_+^2\d z.
  \end{align}
  Since we have $u\ge k_{i+1}$ on the domains
  of integration of the integrals in~\eqref{caccio-means}, we can use the estimate
  \begin{equation*}
    u-k_i\ge k_{i+1}-k_i=\tfrac{k}{2^{i+1}}.
  \end{equation*}
  Therefore, for every exponent $r<m:=\frac{p_\alpha(n+2)}{n}$, we obtain the bound
  \begin{equation}\label{r->m}
    \mint_{Q_i}\big(u-k_{i+1}\big)_+^{r}\,\dz
    \le
    \mint_{Q_i}\big(u-k_{i+1}\big)_+^{r}\Big[\tfrac{2^{i+1}}{k}(u-k_i)_+\Big]^{m-r}\,\dz
    \le
    \frac{c2^{i(m-r)}}{k^{m-r}}\mint_{Q_i}\big(u-
    k_i\big)_+^{m}\,\dz.
  \end{equation}
  Applying this estimate with the exponents $r\in\{\beta',\gamma,2,q_\beta\}$, we get 
\begin{align}\label{caccio-m}
      &\sup_{t\in (-\tilde\sigma_i,0)}\mint_{\widetilde
      B_i\times\{t\}}\frac{\big(u-k_{i+1}\big)_+^2}{\sigma}\,\d x
      +
      \AA^{-1}\bigg(\mint_{\widetilde
        Q_i}\big|D\big(u-k_{i+1}\big)_+\big|^{p_\alpha}\d x\d t\bigg)^{\frac{\alpha+1}{\alpha}}
        \cr
      &\qquad\le
      c\mu^{q-1}2^i\BB
    \left( \frac{2^{i(m-\beta')}}{\rho^{\beta'}k^{m-\beta'}} \mint_{Q_i}\big(u-k_i\big)_+^m\,\dz\right)^\frac{1}{\beta'}\cr
     &\qquad\qquad+
     \Bigg[c2^i\AA^{\frac{q-1}{p}}\BB
     \left(\frac{2^{i(m-\gamma)}}{\rho^\gamma k^{m-\gamma}}\mint_{Q_i}\big(u-k_i\big)_+^m\,\dz\right)^{\frac{1}{\gamma}}\Bigg]^{\frac{p}{p+1-q}}\cr &\qquad\qquad+
     \frac{c\varepsilon}{\rho^{q_\beta}}
   \frac{2^{im}}{k^{m-q_\beta}}\mint_{Q_i}\big(u-
    k_i\big)_+^{m}\,\dz +
     c\frac{2^{i(m-1)}}{\sigma k^{m-2}}\mint_{Q_i} \big(u-k_i\big)_+^m\,\dz\cr
     &\qquad=:\mathrm{R}_1.
\end{align}
Moreover, applying~\eqref{r->m} with $r=p_\alpha$, we obtain
\begin{equation}\label{lower-order-m}
     \mint_{\widetilde Q_i}\frac{2^{ip_\alpha}(u-k_{i+1})_+^{p_\alpha}}{\rho^{p_\alpha}}\,\dz
     \le
     \frac{c2^{im}}{\rho^{p_\alpha}k^{m-p_\alpha}}\mint_{Q_i} \big(u-k_i\big)_+^m\,\dz
     =:\mathrm{R}_2.
\end{equation}  
  Next, we apply Lemma \ref{lem:inter} to the functions $v:=\varphi (u-k_{i+1})_+$ for some cut-off function $\varphi(x)\in C^\infty_0(\tilde B_i,[0,1])$ with $\varphi=1$ in $B_{i+1}$ and $|D\varphi|\le\frac{c2^i}{\rho}$. In the next step, the result is estimated using~\eqref{caccio-m} and \eqref{lower-order-m}. This yields the estimate
  \begin{align*}
    &\mint_{Q_{i+1}}\big(u-k_{i+1}\big)_+^m\,\dz\\
    &\qquad\le
    c(\rho^n\sigma)^{\frac{p_\alpha}{n}}\left(\sup_{t\in(-\tilde\sigma_i,0)}\mint_{\widetilde B_i\times\{t\}}\frac{(u-k_{i+1})_+^2}{\sigma}\,\dx\right)^{\frac{p_\alpha}{n}}\mint_{\widetilde Q_i}\left(\left|D(u-k_{i+1})_+\right|^{p_\alpha}+\frac{2^{ip_\alpha}(u-k_{i+1})_+^{p_\alpha}}{\rho^{p_\alpha}}\right)\,\dz\cr
    &\qquad\le
    c(\rho^n\sigma)^{\frac{p_\alpha}{n}}\mathrm{R}_1^{\frac{p_\alpha}{n}}\Big(\AA^{\frac\alpha{\alpha+1}} \mathrm{R}_1^{\frac{\alpha}{\alpha+1}} +\mathrm{R}_2\Big) \cr
    &\qquad\le
    c(\rho^n\sigma)^{\frac{p_\alpha}{n}}\AA^{\frac\alpha{\alpha+1}}\Big(\mathrm{R}_1^{\frac{n+p}{n}\frac{\alpha}{\alpha+1}}
    +\big(\AA^{-\frac{\alpha}{\alpha+1}}\mathrm{R}_2\big)^{\frac{n+p}{n}}\Big).
  \end{align*}

  In the last step, we applied Young's inequality with exponents $\frac{n+p}{p}$ and $\frac{n+p}{n}$. Recalling the definitions of $\mathrm{R}_1$ and $\mathrm{R}_2$, we arrive at 
  \begin{align*}
     \mint_{Q_{i+1}}&\big(u-k_{i+1}\big)_+^m\,\dz\cr
     &\le
     c\lambda^i(\rho^n\sigma)^{\frac{p_\alpha}{n}}\AA^{\frac{\alpha}{\alpha+1}}\Bigg[
     \left(\frac{\BB^{\beta'}\mu^{(q-1)\beta'}}{\rho^{\beta'}k^{m-\beta'}} \mint_{Q_i}\big(u-k_i\big)_+^m\,\dz\right)^{\frac{1}{\beta'}\frac{n+p}{n}\frac{\alpha}{\alpha+1}}\cr
     &\qquad\qquad\qquad\qquad\quad+
     \left(\frac{\big(\AA^{\frac{q-1}{p}}\BB\big)^\gamma}{\rho^\gamma k^{m-\gamma}}\mint_{Q_i}\big(u-k_i\big)_+^m\,\dz\right)^{\frac{1}{\gamma}\frac{p}{p+1-q}\frac{n+p}{n}\frac{\alpha}{\alpha+1}}
     \cr
     &\qquad\qquad\qquad\qquad\quad+
      \bigg(\frac{c\varepsilon}{\rho^{q_\beta}}
   \frac{1}{k^{m-q_\beta}}\mint_{Q_i}\big(u-
    k_i\big)_+^{m}\,\dz \bigg)^{\frac{n+p}{n}\frac{\alpha}{\alpha+1}}\cr
     &\qquad\qquad\qquad\qquad\quad+
     \bigg(\frac{1}{\sigma k^{m-2}}\mint_{Q_i} \big(u-k_i\big)_+^m\,\dz\bigg)^{\frac{n+p}{n}\frac{\alpha}{\alpha+1}}\cr
     &\qquad\qquad\qquad\qquad\quad+
     \bigg(\frac{1}{\AA^{\frac{\alpha}{\alpha+1}}\rho^{p_\alpha}k^{m-p_\alpha}}\mint_{Q_i} \big(u-k_i\big)_+^m\,\dz\bigg)^{\frac{n+p}{n}}
     \Bigg]\\
     &=:c\lambda^i(\rho^n\sigma)^{\frac{p_\alpha}{n}}\AA^{\frac{\alpha}{\alpha+1}}\big(\mathrm{I}+\mathrm{II}+\mathrm{III}+\mathrm{IV}+\mathrm{V}\big),
  \end{align*}
  for a constant $\lambda>1$ that depends on $n,p,q,\alpha$ and $\beta$.
  By the choice of $k$ in~\eqref{choice-k}, we have 
  \begin{equation}\label{choice-k-1}
      \mu\le\bigg(\frac{\AA\BB k}{\rho}\bigg)^{\frac{1}{p+1-q}}
      \qquad\mbox{and}\qquad
      k^m\ge c\,\mint_{Q_i} \big(u-k_i\big)_+^m\,\dz,
  \end{equation}
  and therefore
  \begin{align*}
    \mathrm{I}
    &\le
    \bigg[\BB^{\beta'}\big(\AA\BB\big)^{\frac{(q-1)\beta'}{p+1-q}}\bigg(\frac{k}{\rho}\bigg)^{\frac{p\beta'}{p+1-q}}\frac{1}{k^m}
    \mint_{Q_i}\big(u-k_i\big)_+^m\,\dz\bigg]^{\frac{1}{\beta'}\frac{n+p}{n}\frac{\alpha}{\alpha+1}}\cr
    &\le
    c\left[\bigg(\frac{\AA^{\frac{q-1}{p}}\BB\,k}{\rho}\bigg)^\gamma\frac{1}{k^m}\mint_{Q_i}\big(u-k_i\big)_+^m\,\dz\right]^{\frac{1}{\gamma}\frac{p}{p+1-q}\frac{n+p}{n}\frac{\alpha}{\alpha+1}}
    =c\mathrm{II}.
  \end{align*}
  In the last step, we used the fact $\frac{1}{\beta'}\ge \frac{1}{\gamma}\frac{p}{p+1-q}$, which is a consequence of~\eqref{bounds-gamma}.
  In order to estimate $\mathrm{III}$, we choose 
  \begin{equation}\label{condition-eps}
     0\le\varepsilon\le\left(\AA^{\frac{q-1}{p}}\BB\right)^\frac{p}{p+1-  q}\left(\frac{k}{\rho}\right)^{\frac{p}{p+1-q}-q_\beta},
  \end{equation}    
  which specifies assumption~\eqref{assumption-eps}.
  For such values of $\varepsilon$, we can estimate 
  \begin{eqnarray}
   \mathrm{ III} &\le&  \bigg[\bigg(\frac{\AA^{\frac{q-1}{p}}\BB k}{\rho}\bigg)^{\frac{p}{p+1-q}}\frac{1}{k^m}\mint_{Q_i} \big(u-k_i\big)_+^m\,\dz\bigg]^{\frac{n+p}{n}\frac{\alpha}{\alpha+1}}\cr &\le&
     c\left[\bigg(\frac{\AA^{\frac{q-1}{p}}\BB\,k}{\rho}\bigg)^\gamma\frac{1}{k^m}\mint_{Q_i}\big(u-k_i\big)_+^m\,\dz\right]^{\frac{1}{\gamma}\frac{p}{p+1-q}\frac{n+p}{n}\frac{\alpha}{\alpha+1}}
     =
     c\mathrm{II},
  \end{eqnarray}
  since $\frac{1}{\gamma}\frac{p}{p+1-q}\le1$ according to~\eqref{bounds-gamma}.
  Similarly, we use the fact 
  \begin{equation*}
      \frac{k^2}{\sigma}\le \frac{1}{\AA}\bigg(\frac{\AA\BB k}{\rho}\bigg)^{\frac{p}{p+1-q}}
  \end{equation*}
  and \eqref{choice-k-1}$_2$ to estimate
  \begin{align*}
     \mathrm{IV}
     &\le
     \bigg[\frac{1}{\AA}\bigg(\frac{\AA\BB k}{\rho}\bigg)^{\frac{p}{p+1-q}}\frac{1}{k^m}\mint_{Q_i} \big(u-k_i\big)_+^m\,\dz\bigg]^{\frac{n+p}{n}\frac{\alpha}{\alpha+1}}\cr
     &\le
     c\left[\bigg(\frac{\AA^{\frac{q-1}{p}}\BB\,k}{\rho}\bigg)^\gamma\frac{1}{k^m}\mint_{Q_i}\big(u-k_i\big)_+^m\,\dz\right]^{\frac{1}{\gamma}\frac{p}{p+1-q}\frac{n+p}{n}\frac{\alpha}{\alpha+1}}
     =
     c\mathrm{II},
  \end{align*}
  where we used again $\frac{1}{\gamma}\frac{p}{p+1-q}\le1$.
  For the estimate of $\mathrm{V}$, we use the fact
  \begin{equation*}
    \bigg(\frac{k}{\rho}\bigg)^{q-p}\ge\frac{1}{\AA\BB}  
  \end{equation*}
  and~\eqref{choice-k-1}$_2$ to deduce 
  \begin{align*}
      \mathrm{V}
      &=
      \bigg[\frac{1}{\AA^{\frac{\alpha}{\alpha+1}}}
      \bigg(\frac{k}{\rho}\bigg)^{\frac{p_\alpha}{p+1-q}}
      \bigg(\frac{k}{\rho}\bigg)^{-(q-p)\frac{p_\alpha}{p+1-q}}
      \frac{1}{k^m}\mint_{Q_i} \big(u-k_i\big)_+^m\,\dz\bigg]^{\frac{n+p}{n}}\cr
     &\le
     c\left[\bigg(\frac{\AA^{\frac{q-1}{p}}\BB k}{\rho}\bigg)^\gamma\frac{1}{k^m}\mint_{Q_i}\big(u-k_i\big)_+^m\,\dz\right]^{\frac{1}{\gamma}\frac{p}{p+1-q}\frac{n+p}{n}\frac{\alpha}{\alpha+1}}
     =
     c\mathrm{II}.
  \end{align*}
  In the last step, we used the estimate $1\ge \frac{1}{\gamma}\frac{p}{p+1-q}\frac{\alpha}{\alpha+1}$, which follows from~\eqref{bounds-gamma}. Collecting the estimates, we arrive at the bound
  \begin{align*}
     \mint_{Q_{i+1}}\big(u-k_{i+1}\big)_+^m\,\dz
     &\le
     c\lambda^i(\rho^n\sigma)^{\frac{p_\alpha}{n}}\AA^{\frac{\alpha}{\alpha+1}}\cdot[\mathrm{II}]\cr
     &=
     c\lambda^i(\rho^n\sigma)^{\frac{p_\alpha}{n}}
     \AA^{\frac{\alpha}{\alpha+1}}\Bigg[\left(\frac{\big(\AA^{\frac{q-1}{p}}\BB\big)^\gamma}{\rho^\gamma k^{m-\gamma}}\mint_{Q_i}\big(u-k_i\big)_+^m\,\dz\right)^{\frac{1}{\gamma}\frac{p}{p+1-q}\frac{n+p}{n}\frac{\alpha}{\alpha+1}}\Bigg].
  \end{align*}
  Abbreviating
  \begin{equation*}
      X_i:=\mint_{Q_i}\big(u-k_i\big)_+^m\,\d z,\quad i\in\N_0,
  \end{equation*}
  and
  \begin{equation}\label{def-kappa}
    \kappa:=\frac1\gamma \frac{p}{p+1-q}\frac{n+p}{n}\frac{\alpha}{\alpha+1}-1>0,
  \end{equation}

  the preceding estimate becomes 
  \begin{equation*}
      X_{i+1}\le \frac{c\Big[\frac{\sigma}{\AA}\big(\frac{\AA\BB}{\rho}\big)^{\frac{p}{p+1-q}}\Big]^{\frac{p_\alpha}{n}}\Big(\frac{\AA\BB}{\rho^{q-p}}\Big)^{\frac{p_\alpha}{p+1-q}}}{k^{(m-\gamma)(1+\kappa)}}\lambda^i X_i^{1+\kappa}
  \end{equation*}
  for every $i\in\N_0$.
 The exponent $\kappa$ is positive since~\eqref{upper-bound-gamma} implies $\kappa>\frac{q-p}{p+1-q}\ge0$. Therefore, Lemma~\ref{lem:geom-conv} is applicable if we have 
  \begin{equation*}
      \mint_{Q_{2\rho,2\sigma}}u_+^m\,\d z
      =
      X_0\le c\Bigg(\bigg[\frac{\sigma}{\AA}\bigg(\frac{\AA\BB}{\rho}\bigg)^{\frac{p}{p+1-q}}\bigg]^{-\frac{p_\alpha}{n}}\bigg(\frac{\AA\BB}{\rho^{q-p}}\bigg)^{-\frac{p_\alpha}{p+1-q}} k^{(m-\gamma)(1+\kappa)}\Bigg)^{\frac1\kappa},
  \end{equation*}
  or equivalently,
  \begin{equation*}
      k^{(m-\gamma)(1+\kappa)}
      \ge 
      c\bigg[\frac{\sigma}{\AA}\bigg(\frac{\AA\BB}{\rho}\bigg)^{\frac{p}{p+1-q}}\bigg]^{\frac{p_\alpha}{n}}
      \bigg(\frac{\AA\BB}{\rho^{q-p}}\bigg)^{\frac{p_\alpha}{p+1-q}}
      \bigg(\mint_{Q_{2\rho,2\sigma}}u_+^m\,\d z\bigg)^\kappa,
  \end{equation*}
  with a constant $c$ depending on $n,p,q,\alpha$ and $\beta$.
  This is satisfied for our choice of $k$ in~\eqref{choice-k} if we define
  \begin{align*}
      \vartheta_1:=\frac{p_\alpha}{n(m-\gamma)(1+\kappa)},\quad
      \vartheta_2:=\frac{p_\alpha}{(p+1-q)(m-\gamma)(1+\kappa)}\quad\mbox{and}\quad
      \vartheta_3:=\frac{\kappa}{(m-\gamma)(1+\kappa)}.
  \end{align*}
  Note that these parameters are well defined since we have $\gamma<m$ by \eqref{upper-bound-gamma}. At this point we rely on the strict inequality in assumption~\eqref{gap}.
  With these choices of the parameters, Lemma~\ref{lem:geom-conv} implies
  \begin{equation*}
      \mint_{Q_{\rho,\sigma}}(u-k)_+^m\,\d z
      =
      \lim_{i\to\infty}\mint_{Q_i}\big(u-k_i\big)_+^m\,\d z=0,
  \end{equation*}
  proving $u\le k$ a.e. in $Q_{\rho,\sigma}$. This proves the asserted estimate.
\end{proof}

\section{Existence of a locally bounded variational solution}\label{sec:exist}

 In this section, we apply the local sup-bound from the preceding
 section to construct locally bounded variational solutions to
 Cauchy-Dirichlet problems, i.e. to give the 

\begin{proof}[Proof of Theorem \ref{thm:exist}]
  The proof is divided in several steps.

  \medskip
  
\emph{Step 1: Regularization. }
For some sequence $\eps_i\in(0,1]$ with $\eps_i\downarrow0$,
we define regularized integrands $f_i$ as
\begin{equation*}
  f_i(x,t,\xi):= f(x,t,\xi)+\eps_i|\xi|^{q_\beta}
  \qquad\mbox{for }\xi\in\R^{n},\ (x,t)\in\Omega_T.
\end{equation*}
As a consequence of assumptions \eqref{fip0} --
\eqref{fip2} on $f$, we have 
\begin{align}\label{feps0}
  0 \le
  f_i(x,t,\xi)
  &\le
  b(x,t)(\mu^2+|\xi|^2)^{\frac q2}+\eps_i|\xi|^{q_\beta},\\
\label{feps1}
|D_\xi f_i(x,t,\xi)|&\le b(x,t)(\mu^2+|\xi|^2)^{\frac{q-1}{2}}+c(q,\beta)\eps_i|\xi|^{q_\beta-1},\\
 \label{feps2}
  \langle D_\xi f_i(x,t,\xi),\xi\rangle
  &\ge
  a(x,t)(\mu^2+|\xi|^2)^{\frac{p-2}{2}}|\xi|^2+c(q,\beta)\eps_i|\xi|^{q_\beta},
\end{align}
for each $i\in\N$, a.e. $(x,t)\in\Omega_T$, and any $\xi\in\R^{n}$.
In particular, \eqref{feps1} and Young's inequality imply
\begin{equation*}
  |D_\xi f_i(x,t,\xi)|\le (1+c\eps_i)|\xi|^{q_\beta-1} + h(x,t)
\end{equation*}
with 
\begin{equation*}
    h(x,t):=c\Big(b(x,t)+b(x,t)^{\frac{q\beta-\beta+1}{q}}\Big)\in L^{q_\beta'}(\Omega).
\end{equation*}
Therefore and because of $a(x,t)\ge0$, assumptions \eqref{feps1} and \eqref{feps2} imply that the regularized functions $D_\xi f_i$ satisfy standard
$q_\beta$-growth and ellipticity conditions. Furthermore, the
initial data satisfy \eqref{initial-condition}, and assumption \eqref{lateral-condition-f} for the lateral boundary values implies 
\begin{equation*}
   g\in L^{q_\beta}(0,T;W^{1,q_\beta}(\Omega))
    \qquad\mbox{and}\qquad
    \partial_t g\in L^{q_\beta'}(0,T;W^{-1,q_\beta'}(\Omega)),
  \end{equation*}
  since the fact $p_\alpha<q_\beta$ implies
  $\gamma>q_\beta$. Therefore, classical theory
(see e.g. \cite[Section III.4]{Showalter}) yields the existence of unique solutions $u_i\in
g+L^{q_\beta}(0,T;W_0^{1,q_\beta}(\Omega))$ 
of the Cauchy-Dirichlet problems
\begin{equation}\label{diri}\begin{cases}
\partial_tu_i-\dive D_\xi f_i(x,t,Du_i)=0& \mathrm{in} \ \Omega_T,\\
u_i=g & \mathrm{on} \ \partial_{\mathrm{par}} \Omega_T,
\end{cases}
\end{equation}
for any $i\in\N$.

\medskip

\emph{Step 2: Energy bounds. }
As a direct consequence of the parabolic system in~\eqref{diri} and
assumption \eqref{lateral-condition-f} on the boundary values, we have 
$$\partial_tg,\ \partial_tu_i\in L^{q_\beta'}\big(0,T;W^{-1,q_\beta'}(\Omega)\big).$$ 
Hence, for any $t_0\in(0,T)$ we may use
$\varphi=(u_i-g)\chi_{(0,t_0)}\in L^{q_\beta}(0,T;W_0^{1,{q_\beta}}(\Omega))$
as test function in \eqref{diri}, with the result
\begin{align}\label{comparison1}
\mathrm{I}&+\mathrm{II}\\\nonumber
&:=
\int_0^{t_0}\langle \partial_t(u_i-g), u_i-g\rangle_{W^{-1,q_\beta'}}\,\dt
+\int_{\Omega_{t_0}}\langle
D_\xi f_i(x,t,Du_i),Du_i\rangle\,\d z\\
&
=-\int_0^{t_0}\langle \partial_tg, u_i-g\rangle_{W^{-1,q_\beta'}}\,\dt
+\int_{\Omega_{t_0}}\langle
D_\xi f_i(x,t,Du_i),Dg\rangle\,\d z\nonumber\\\nonumber
&=:\mathrm{III}+\mathrm{IV},
\end{align}
where $\langle\cdot,\cdot\rangle_{W^{-1,q_\beta'}}$ denotes the duality pairing
between $W^{-1,q_\beta'}(\Omega)$ and $W^{1,q_\beta}_0(\Omega)$.
The first integral can be re-written in a standard way
(cf. \cite[Prop. III.1.2]{Showalter}) as
\begin{align}\label{comparison2}
  \mathrm{I}&=\int_0^{t_0}\langle \partial_t(u_i-g),u_i-g\rangle_{W^{-1,q_\beta'}}\,\dt\cr
  &\qquad =
  \frac12\int_0^{t_0}\partial_t\|u_i-g\|_{L^2(\Omega)}^2\,\dt
  =
  \frac12\int_{\Omega\times\{t_0\}}|u_i-g|^2\,\dx,
\end{align}
where we used the initial condition $u_i=g$ at $t=0$
according to \eqref{diri}.
For the estimate of the second integral on the left-hand side of
\eqref{comparison1}, we use \eqref{feps2} to infer
\begin{align}\label{comparison2.5}
\mathrm{II}&=\int_{\Omega_{t_0}}\langle
D_\xi f_i(x,t,Du_i),Du_i\rangle\,\d z\\\nonumber
&\qquad\ge
\int_{\Omega_{t_0}}a(x,t)(\mu^2+|Du_i|^2)^{\frac{p-2}2}|Du_i|^2\,\d z
+c(q,\beta)\varepsilon_i\int_{\Omega_{t_0}}|Du_i|^{q_\beta}\,\d z\\\nonumber
&\qquad\ge \|a^{-1}\|_{L^\alpha}^{-1} \left(\int_{\Omega_{t_0}}|Du_i|^{p_\alpha}\,\d z\right)^{\frac{\alpha+1}{\alpha}}
+
c(q,\beta)\varepsilon_i\int_{\Omega_{t_0}}|Du_i|^{q_\beta}\,\d z.
\end{align}
For the estimate of $\mathrm{III}$, we use the fact $\partial_t g\in
L^{p_{\alpha}'}(0,T;W^{-1,p_{\alpha}'}(\Omega))$ that holds according to~\eqref{lateral-condition-f}. This allows us to estimate
\begin{align}\label{comparison2.6}
  \mathrm{III}
  &\le
  \int_0^{t_0}\|Du_i-Dg\|_{L^{p_{\alpha}}(\Omega)}\,\|\partial_t g\|_{W^{-1,p_{\alpha}'}(\Omega)}\,\dt\\\nonumber
  &\le
  \|a^{-1}\|_{L^\alpha}^{-\frac{1}{p}}\|a^{-1}\|_{L^\alpha}^{\frac{1}{p}}
  \left(\int_{\Omega_{t_0}}|Du_i-Dg|^{p_\alpha}\,\d z\right)^{\frac{1}{p_\alpha}}
  \|\partial_t
    g\|_{L^{p_{\alpha}'}\mbox{-}W^{-1,p_{\alpha}'}}\\\nonumber
  &\le 
    \frac{\|a^{-1}\|_{L^\alpha}^{-1}}{4}\left(\int_{\Omega_{t_0}}|Du_i|^{p_{\alpha}}\,\d z\right)^{\frac{\alpha+1}{\alpha}}
    +\frac{\|a^{-1}\|_{L^\alpha}^{-1}}{4}\left(\int_{\Omega_{t_0}}|Dg|^{p_{\alpha}}\,\d z\right)^{\frac{\alpha+1}{\alpha}}\\\nonumber
    &\qquad+ c(p,\alpha)\|a^{-1}\|_{L^\alpha}^{\frac{1}{p-1}}\|\partial_t g\|^{p'}_{L^{p_{\alpha}'}\mbox{-}W^{-1,p_{\alpha}'}}.
\end{align}
Finally, using \eqref{feps1} and H\"older's inequality with the three
exponents $\beta$, $\frac{p_\alpha}{q-1}$ and $\gamma$ and then Young's inequality, we deduce
\begin{align}\label{comparison 2.8}
  \mathrm{IV}
  &\le
    \int_{\Omega_{t_0}} |D_\xi f_i(x,t,Du_i)| |Dg|\,\dz\\\nonumber
  &\le
    \int_{\Omega_{t_0}}b(x,t)(\mu+|Du_i|)^{q-1}|Dg|\,\d z
    +c(q,\beta)\eps_i\int_{\Omega_{t_0}}|Du_i|^{q_\beta-1}|Dg|\,\d
    z\\\nonumber
  &\le
    c(q)\|b\|_{L^\beta} \|a^{-1}\|_{L^\alpha}^{\frac{q-1}{p}}\|a^{-1}\|_{L^\alpha}^{-\frac{q-1}{p}}\bigg(\int_{\Omega_{t_o}}|Du_i|^{p_\alpha}\d
    z\bigg)^{\frac{q-1}{p_\alpha}}\bigg(\int_{\Omega_{t_o}}|Dg|^{\gamma}\d
    z\bigg)^{\frac1\gamma}\\\nonumber
    &\qquad+c(q)\mu^{q-1}\|b\|_{L^\beta}\bigg(\int_{\Omega_{t_o}}|Dg|^{\beta'}\d
    z\bigg)^{\frac{1}{\beta'}}
     +c(q,\beta)\eps_i\bigg(\int_{\Omega_{t_o}}|Du_i|^{q_\beta}\d
    z\bigg)^{\frac{1}{q_\beta'}}\bigg(\int_{\Omega_{t_o}}|Dg|^{q_\beta}\d
    z\bigg)^{\frac{1}{q_\beta}}
  \\\nonumber
  &\le
  \frac{\|a^{-1}\|_{L^\alpha}^{-1}}{4}\left(\int_{\Omega_{t_0}}|Du_i|^{p_{\alpha}}\,\d z\right)^{\frac{\alpha+1}{\alpha}}+\frac{c(q,\beta)\varepsilon_i}{2}\int_{\Omega_{t_0}}|Du_i|^{q_\beta}\,\d z
 \\\nonumber
  &\quad+ c(p,q)\|b\|_{L^\beta}^{\frac{p}{p+1-q}}\|a^{-1}\|_{L^\alpha}^{\frac{q-1}{p+1-q}}\|Dg\|^{\frac{p}{p+1-q}}_{L^\gamma}
    +
    c(q)\mu^{q-1}\|b\|_{L^\beta}\|Dg\|_{L^{\beta'}}
    +c(q,\beta)\varepsilon_i\|Dg\|_{L^{q_\beta}}^{q_\beta}
  .
\end{align}

We note that the three last terms are finite by~\eqref{lateral-condition-f} and \eqref{upper-bound-gamma}. 
Plugging \eqref{comparison2}, \eqref{comparison2.5},
\eqref{comparison2.6} and \eqref{comparison 2.8} into
\eqref{comparison1} and reabsorbing the integrals involving the spatial
derivatives, we arrive at 
 \begin{eqnarray}\label{comparison3}
&&\frac12\int_{\Omega\times\{t_0\}}|u_i-g|^2\dx+ \frac{\|a^{-1}\|_{L^\alpha}^{-1}}{2}\left(\int_{\Omega_{t_0}}|Du_i|^{p_{\alpha}}\,\d z\right)^{\frac{\alpha+1}{\alpha}}
+\frac{c(q,\beta)\varepsilon_i}{2}\int_{\Omega_{t_0}}|Du_i|^{q_\beta}\,\d z\\\nonumber
&&\qquad\quad\le  
c_{a,b}\Big(\|\partial_t g\|^{p'}_{L^{p_{\alpha}'}\mbox{-}W^{-1,p_{\alpha}'}}
   +\|Dg\|_{L^{p_\alpha}}^p
   +\|Dg\|^{\frac{p}{p+1-q}}_{L^\gamma}
   +\mu^{q-1}\|Dg\|_{L^{\beta'}}
  +
  \varepsilon_i\|Dg\|_{L^{q_\beta}}^{q_\beta}\Big)
\end{eqnarray}
for every $t_0\in(0,T]$, with a constant $c_{a,b}$ depending on
$p,q,\alpha,\beta,\|a^{-1}\|_{L^\alpha}$ and $\|b\|_{L^\beta}$. Taking the supremum over $t_0\in(0,T]$, we
arrive at the energy bound
\begin{align}\label{energy-bound}
  \sup_{t\in[0,T]}\int_{\Omega\times\{t\}}|u_i|^2\dx+\left(\int_{\Omega_{T}}|Du_i|^{p_{\alpha}}\,\d z\right)^{\frac{\alpha+1}{\alpha}}+\varepsilon_i\int_{\Omega_{T}}|Du_i|^{q_\beta}\,\d z
  \le
  c_{a,b} \big(M_g+
  \varepsilon_i\|Dg\|_{L^{q_\beta}}^{q_\beta}
  \big)
\end{align}
for every $i\in\N$, where we used the abbreviation 
\begin{equation}\label{abb}
  M_{g}:= \|\partial_t
  g\|^{p'}_{L^{p_{\alpha}'}\mbox{-}W^{-1,p_{\alpha}'}}
  +\|g\|_{L^{p_\alpha}-W^{1,p_\alpha}}^p
  +\|Dg\|^{\frac{p}{p+1-q}}_{L^\gamma}
   +\mu^{q-1}\|Dg\|_{L^{\beta'}}
  +\|g\|_{L^2-L^\infty}^2.
\end{equation}
Moreover, applying Poincar\'e's inequality to $(u_i-g)(\cdot,t)$ for
a.e. $t\in(0,T)$, we deduce
\begin{align}
  \label{energy-bound-2}
  \int_{\Omega_T}|u_i|^{p_{\alpha}}\,\d z
  &\le
  c(p,\alpha,\mathrm{diam}(\Omega))\int_{\Omega_T}|Du_i-Dg|^{p_{\alpha}}\,\d z+c(p,\alpha)\int_{\Omega_T}|g|^{p_{\alpha}}\,\d z\\\nonumber
  &\le
  C \big(M_g+
  \varepsilon_i\|Dg\|_{L^{q_\beta}}^{q_\beta}
    \big)^{\frac{\alpha}{\alpha+1}}
\end{align}
for every $i\in\N$, with a constant $C$ depending on 
$p,q,\alpha,\beta,\|a^{-1}\|_{L^\alpha},\|b\|_{L^\beta}$ and $\mathrm{diam}(\Omega)$. 
In view of the preceding energy bounds, after passing to a subsequence
we can achieve weak convergence to some limit map $ u\in
g+L^{p_\alpha}(0,T;W^{1,p_\alpha}_0(\Omega))$ in the sense
\begin{equation}
  \label{weak-converge}
  \left\{
  \begin{array}{cl}
    Du_i\wto D u&\mbox{weakly in }L^{p_{\alpha}}(\Omega_T,\R^{n}),\\[0.5ex]
    u_i\wto  u&\mbox{weakly in }L^{p_{\alpha}}(\Omega_T),\\[0.5ex]
    u_i\wsto u&\mbox{weakly* in $L^\infty(0,T;L^2(\Omega)),$}
  \end{array}\right.
\end{equation}
in the limit $i\to\infty$. The lower semicontinuity of the norms with respect to these convergences and the bounds \eqref{energy-bound} and \eqref{energy-bound-2} imply
\begin{equation}\label{energy-bound-limit}
  \sup_{t\in[0,T]}\int_{\Omega\times\{t\}}|u|^2\dx+\left(\int_{\Omega_{T}}|Du|^{p_{\alpha}}+|u|^{p_\alpha}\,\d z\right)^{\frac{\alpha+1}{\alpha}}
  \le CM_g.
\end{equation}

\emph{Step 3: Weak continuity in time. }
Testing the weak formulation of the
parabolic system~\eqref{diri} with a function $\varphi\in
C^\infty_0(\Omega_T)$ and exploiting the growth condition
\eqref{feps1}, we infer
\begin{align*}
  \bigg|\int_{\Omega_T}u_i\partial_t\varphi\,\d z\bigg|
  &\le
  \int_{\Omega_T}b(x,t)(1+|Du_i|)^{q-1}|D\varphi|\,\d z+c\varepsilon_i\int_{\Omega_T}|Du_i|^{q_\beta-1}|D\varphi|\,\d z.
\end{align*}
Now, since assumption \eqref{gap} implies $q-1<p_\alpha\frac{\beta-1}{\beta}$, we may use H\"older's inequality to obtain 
\begin{align*}
  \bigg|\int_{\Omega_T}u_i\partial_t\varphi\,\d z\bigg|
  &\le
  \|b\|_{L^\beta}\,|\spt\varphi|^{\frac{\beta-1}\beta-\frac{q-1}{p_\alpha}}
  \bigg(\int_{\Omega_T}(1+|Du_i|)^{p_\alpha}\,\d z\bigg)^{\frac{q-1}{p_\alpha}}\|D\varphi\|_{L^\infty(\Omega_T)}\\
  &\qquad+
  c\varepsilon_i\,|\spt\varphi|^{\frac{1}{q_\beta}}
  \bigg(\int_{\Omega_T}|Du_i|^{q_\beta}\,\d z\bigg)^{\frac{q_\beta-1}{q_\beta}}\|D\varphi\|_{L^\infty(\Omega_T)}\\
  &\le
  C \Big(|\spt\varphi|^{\frac{\beta-1}\beta-\frac{q-1}{p_\alpha}}+|\spt\varphi|^{\frac{1}{q_\beta}}\Big)\|D\varphi\|_{L^\infty(\Omega_T)}.
\end{align*}
In the last line we used the energy estimate~\eqref{energy-bound}, and the constant $C$ depends on $n,p,q,\alpha,\beta, \Omega,M_{g}, \|a^{-1}\|_{L^\alpha}$ and $\|b\|_{L^\beta}$.  
Now we choose a test function of the form
$\varphi(x,t)=\chi_\delta(t)\psi(x)$, where $\psi\in
C^\infty_0(\Omega)$, and $\chi_\delta$ is a suitable smooth function
that approximates $\chi_{(t_1,t_2)}$ for times $0\le t_1<t_2\le T$. Letting
$\delta\downarrow0$, the preceding estimate implies
\begin{align*}
  \bigg|\int_\Omega(u_i(x,t_2)-u_i(x,t_1))\psi(x)\,\dx\bigg|
  &\le
  C\Big(|t_2-t_1|^{\frac{\beta-1}\beta-\frac{q-1}{p_\alpha}}+|t_2-t_1|^{\frac{1}{q_\beta}}\Big)\|D\psi\|_{L^\infty(\Omega)}\\
  &\le
  C\Big(|t_2-t_1|^{\frac{\beta-1}\beta-\frac{q-1}{p_\alpha}}+|t_2-t_1|^{\frac{1}{q_\beta}}\Big)\|\psi\|_{W^{\ell,2}_0(\Omega)}
\end{align*}
for every $i\in\N$ and any $\ell\in\N$ with
$\ell>\frac{n+2}2$, where we used the Sobolev embedding
$W^{\ell,2}_0(\Omega)\subset W^{1,\infty}(\Omega)$ for the last
estimate. 
Therefore, we have
$u_i\in C^0([0,T];W^{-\ell,2}(\Omega))$ with the uniform estimate 
\begin{align*}
  \|u_i(\cdot,t_2)-u_i(\cdot,t_1)\|_{W^{-\ell,2}(\Omega)}\le C\Big(|t_2-t_1|^{\frac{\beta-1}\beta-\frac{q-1}{p_\alpha}}+|t_2-t_1|^{\frac{1}{q_\beta}}\Big)
\end{align*}
for any $t_1,t_2\in[0,T]$. Since, furthermore, the sequence $u_i$ is
bounded in the space $L^\infty(0,T;L^2(\Omega))$ by
\eqref{energy-bound}, 
the compactness result \cite[Thm. A.2]{BDM:pq} implies
\begin{equation}
  \label{slicewise-converge}
  u_i(\cdot,t)\wto  u(\cdot,t)
  \qquad\mbox{weakly in $L^2(\Omega)$ in the limit $i\to\infty$, for every $t\in [0,T]$,} 
\end{equation}
and $ u\in C_w([0,T];L^2(\Omega))$, cf. also \cite[Thm. 2.1]{Strauss}.\\

\emph{Step 4: Application of the sup-bound. }
Let $K\subset\Omega\times(0,T]$ be an arbitrary compact subset and $z_o=(x_o,t_o)\in K$.
For
$\rho:=\min\{1,\tfrac14\dist_{\mathrm{par}}^{(p,q)}(K,\partial_{\mathrm{par}}\Omega_T)\}$
we consider the cylinders
$$
  Q:=B\times\Lambda:=Q_{\rho}^{(p,q)}(z_o)
  \qquad\mbox{and}\qquad
  \widehat Q:=\widehat B\times\widehat\Lambda:= Q_{2\rho}^{(p,q)}(z_o),
$$
as defined in~\eqref{pq-cyl}. By
choice of $\rho$ we have 
$\dist_{\mathrm{par}}^{(p,q)}(\widehat Q,\partial_{\mathrm{par}}\Omega_T)>0$. We apply the
interpolation inequality from Lemma~\ref{lem:inter} to the function
$\varphi u_i$, where $\varphi(x)\in C^\infty_0(B_{2\rho}(x_o))$ denotes a
standard cut-off function with $\varphi\equiv1$ in $B_{\rho}(x_o)$
and $|D\varphi|\le \frac{2}{\rho}$. In this way, we infer
\begin{align*}
    \mint_{Q}\left|u_i\right|^{m}\dz\leq
  c(n,p,\alpha)\left(\sup_{t\in\widehat\Lambda}\,\int_{\widehat B}\left|u_i(x,t)\right|^{2}\dx\right)^{\frac{p_\alpha}{n}}\mint_{\widehat Q}\left(|Du_i|^{p_\alpha}
    +
    \frac{|u_i|^{p_\alpha}}{\rho^{p_\alpha}}\right)\dz,
\end{align*}
where we used the abbreviation $m:=\frac{p_\alpha(n+2)}{n}$. 
The energy bounds \eqref{energy-bound} and \eqref{energy-bound-2} imply that the right-hand side is bounded independently of $i\in\N$. 
More precisely, we obtain 
\begin{equation}
    \label{ui-Lm}
    \mint_Q |u_i|^m\,\dz
    \le
    C\big(M_g+\varepsilon_i\|Dg\|_{L^{q_\beta}}^{q_\beta}\big)^{\frac{n+p}{n}\frac{\alpha}{\alpha+1}}
\end{equation}
  for every $i\in\N$, with a constant $C$ depending on 
$n,p,q,\alpha,\beta,\|a^{-1}\|_{L^\alpha},\|b\|_{L^\beta}$ and $\mathrm{diam}(\Omega)$.
We use the abbreviation 
\begin{align*}
    k_i
    :=
    \mathrm{C}_{a,b}\max\left\{\frac{1}{\rho^{(q-p)\vartheta_2}}\bigg(\mint_{Q} (u_i)_+^m\,\d z\bigg)^{\vartheta_3},\bigg(\mint_{Q} (u_i)_+^m\,\d z\bigg)^{\frac1m},1\right\}.
\end{align*}
This corresponds to the definition of $k$ in Proposition~\ref{prop:apriori}
for the present case $\sigma=\rho^{\frac{p}{p+1-q}}$ and $\rho\le1$. The positive exponents $\vartheta_j=\vartheta_j(n,p,q,\alpha,\beta)$, $j\in\{2,3\}$, and the constant $\mathrm{C}_{a,b}=\mathrm{C}_{a,b}(n,p,q,\alpha,\beta,\mint_Q a^{-\alpha},\mint_Qb^\beta)$ are introduced in Proposition~\ref{prop:apriori}.
The upper bound~\eqref{ui-Lm} and the definition of $k_i$ imply 
\begin{equation}\label{bound-ki}
    0<\mathrm{C}_{a,b}\le k_i\le C_K\big(1+M_g+
  \varepsilon_i\|Dg\|_{L^{q_\beta}}^{q_\beta}
    \big)^\vartheta
    \qquad\mbox{for all }i\in\N,
\end{equation}
with a constant $C_K$ depending on $n,p,q,\alpha,\beta,\mint_Q
a^{-\alpha},\mint_Qb^\beta$, $\mathrm{diam}(\Omega)$ and
$\dist_{\mathrm{par}}^{(p,q)}(K,\partial_{\mathrm{par}}\Omega_T)$ and
an exponent $\vartheta=\vartheta(n,p,q,\alpha,\beta)>0$. 
Because of this and $\varepsilon_i\to0$ as $i\to\infty$, we can find an index $i_o(K)\in\N$ with 
\begin{align}\label{small-eps}
    \varepsilon_i<\widetilde{\mathrm{C}}_{a,b}\left(\frac{k_i}{\rho}\right)^{\frac{p}{p+1-q}-q_\beta}
    \qquad\mbox{for any }i\ge i_o(K),
\end{align}
with the constant $\widetilde{\mathrm{C}}_{a,b}=\widetilde{\mathrm{C}}_{a,b}(n,p,q,\alpha,\beta,\mint_Q a^{-\alpha},\mint_Qb^\beta)$ from Proposition~\ref{prop:apriori}. 
Because of~\eqref{small-eps}, we can apply the sup-bound from Proposition~\ref{prop:apriori} to $u_i$ for $i\ge i_o(K)$, which implies 
\begin{equation*}
  \esssup_{\frac12 Q}u_i\le k_i
  \qquad\mbox{for all }i\ge i_o(K).
\end{equation*}
Recalling that the center of $\frac12Q$ is an arbitrary point $z_o\in
K$ and using~\eqref{bound-ki}, we deduce 
\begin{equation*}
  \esssup_{K}u_i\le C_K\big(1+M_g+
  \varepsilon_i\|Dg\|_{L^{q_\beta}}^{q_\beta}
    \big)^\vartheta
     \qquad\mbox{for all }i\ge i_o(K).
   \end{equation*}
Arguing analogously for $-u_i$ instead of $u_i$, we even obtain
\begin{equation*}
  \esssup_{K}|u_i| \le C_K\big(1+M_g+
  \varepsilon_i\|Dg\|_{L^{q_\beta}}^{q_\beta}\big)^\vartheta
  \qquad\mbox{for all }i\ge i_o(K),
\end{equation*}
and letting $i\to\infty$, we arrive at the bound
\begin{equation}\label{local-bound-u-proof}
  \esssup_{K}|u|
  \le C_K(1+M_g)^\vartheta.
\end{equation}
This yields the local bound \eqref{local-bound-u} stated in Theorem~\ref{thm:exist}.

\medskip

\emph{Step 5: Variational inequality for the limit map. }
For any $\varphi\in L^{q_\beta}(0,\tau;W^{1,q_\beta}_0(\Omega))$, where
$\tau\in(0,T]$, the weak formulation of system~\eqref{diri} and
the convexity of $\xi\mapsto f(x,t,\xi)$ imply
\begin{align*}
  0&=
  \int_0^\tau\langle\partial_t u_i,\varphi\rangle_{W^{-1,q_\beta'}}\,\dt
  +
  \int_{\Omega_\tau}\langle D_\xi f_i(x,t,Du_i),D\varphi\rangle\,\d z\\
  &\le
  \int_0^\tau\langle\partial_t u_i,\varphi\rangle_{W^{-1,q_\beta'}}\,\dt
  +
  \int_{\Omega_\tau}\big[f_i(x,t,Du_i+D\varphi)-f_i(x,t,Du_i)\big]\,\d z.
\end{align*}
We choose $\varphi=v-u_i$ for some comparison map
\begin{equation*}
  v\in g+L^{q_\beta}(0,T;W^{1,q_\beta}_0(\Omega))
  \mbox{\quad with\quad}
  \partial_t v\in L^{p_\alpha'}(0,T;W^{-1,p_\alpha'}(\Omega)).
\end{equation*}
With this choice, the preceding
inequality can be rewritten to 
\begin{align}\label{var-ineq-uk}
  &\int_0^\tau\langle\partial_t v,v-u_i\rangle_{W^{-1,p_\alpha'}}\,\dt
  +
  \int_{\Omega_\tau}\big[f_i(x,t,Dv)-f_i(x,t,Du_i)\big]\,\d z\\\nonumber
  &\qquad\ge
    \int_0^\tau\langle\partial_t
    (v-u_i),v-u_i\rangle_{W^{-1,q_\beta'}}\,\dt\\\nonumber
  &\qquad=
   \tfrac12\int_\Omega |(v-u_i)(\tau)|^2\,\dx
   -
   \tfrac12\int_\Omega |(v-g)(0)|^2\,\dx.
\end{align}
From the weak convergence
\eqref{weak-converge}, \eqref{slicewise-converge}, we deduce
\begin{equation*}
   \int_{\Omega_\tau}f(x,t,Du)\,\d z
   \le
   \liminf_{i\to\infty}\int_{\Omega_\tau}f(x,t,Du_i)\,\d z  
   \le 
   \liminf_{i\to\infty}\int_{\Omega_\tau}f_i(x,t,Du_i)\,\d z
\end{equation*}
by the convexity of $\xi\mapsto f(x,t,\xi)$, and
\begin{equation*}
  \int_\Omega |(v-u)(\tau)|^2\,\dx
  \le
  \liminf_{i\to\infty}\int_\Omega |(v-u_i)(\tau)|^2\,\dx
\end{equation*}
by the weak lower semicontinuity of the $L^2(\Omega)$-norm with respect to weak convergence. 
Moreover, because of $Dv\in L^{q_\beta}(\Omega_T)$ and $\eps_i\downarrow0$ we have 
\begin{equation*}
    \int_{\Omega_\tau}f(x,t,Dv)\,\d z
    =
    \lim_{i\to\infty}\int_{\Omega_\tau}f_i(x,t,Dv)\,\d z.
\end{equation*}
Using this together with the weak convergence~\eqref{weak-converge}, by passing to the limit in \eqref{var-ineq-uk} we deduce
\begin{align}\label{var-ineq-limit}
  &\tfrac12\int_\Omega |(v-u)(\tau)|^2\,\dx
  +
  \int_{\Omega_\tau}f(x,t,Du)\,\d z\\\nonumber
  &\qquad\le
  \int_{\Omega_\tau}f(x,t,Dv)\,\d z
  +   
  \int_0^\tau\langle\partial_t v,v-u\rangle_{W^{-1,p_\alpha'}}\,\dt
  +
  \tfrac12\int_\Omega |(v-g)(0)|^2\,\dx,
\end{align}
for any $\tau\in[0,T]$ and $v\in 
g+L^{q_\beta}(0,\tau;W^{1,q_\beta}_0(\Omega))$ with $\partial_t v\in
L^{p_\alpha'}(0,T;W^{-1,p_\alpha'}(\Omega))$.
This proves that $u$ is a variational solution of~\eqref{diri} in the
sense of Definition~\ref{def:var-sol}. Since it is also
locally bounded according to \eqref{local-bound-u-proof}, the proof of Theorem~\ref{thm:exist} is complete. 

\end{proof}


\begin{thebibliography}{99}

\bibitem{bdgp}
A.~Balci, L.~Diening, R.~Giova, A.~Passarelli di Napoli, 
\textit{Elliptic equations with degenerate weights. }
SIAM J. Math. Anal. 54 (2022), no. 2, 2373--2412.

\bibitem{Bella-Schaeffner:2021}  
P.~Bella, M.~Sch\"affner,
\textit{Local boundedness and Harnack inequality for solutions of linear nonuniformly elliptic equations. }
Comm. Pure Appl. Math. 74 (2021), no. 3, 453--477.

\bibitem{Bella-Schaeffner:2023}  
P.~Bella, M.~Sch\"affner,
\textit{Local boundedness for p-Laplacian with degenerate coefficients}.
Math. Eng. 5 (2023), no. 5, Paper No. 081, 20 pp.  

\bibitem{BDM:pq-equations} 
V.~B\"ogelein, F.~Duzaar and P.~Marcellini,
\textit{Parabolic equations with p,q-growth. }
J. Math. Pures Appl. (9) 100 (2013), no. 4, 535--563.
  
\bibitem{BDM:pq} 
V.~B\"ogelein, F.~Duzaar and P.~Marcellini, 
\textit{Parabolic systems with $p,q$-growth: a variational approach. }  
Arch. Ration. Mech. Anal. 210 (2013), no. 1, 219--267.

\bibitem{BDM:variational} 
V.~B\"ogelein, F.~Duzaar and P.~Marcellini, 
\textit{Existence of evolutionary variational solutions via the calculus of variations}
J. Differential Equations 256 (2014), no. 12, 3912--3942.

\bibitem{CMM}
  G.~Cupini, P.~Marcellini, E.~Mascolo, 
  \textit{Nonuniformly elliptic energy integrals with p,q-growth.}
  Nonlinear Anal. 177 (2018), 312--324.
  
\bibitem{CMMP}
  G.~Cupini, P.~Marcellini, E.~Mascolo, A.~Passarelli di Napoli,
 \textit{Lipschitz regularity for degenerate elliptic
integrals with $(p,q)$-growth.} Adv. Calc. Var.  16, (2023), no.2,  443--465.

\bibitem{DeFilippis}
C.~De Filippis, 
\textit{Gradient bounds for solutions to irregular parabolic equations with $(p,q)$-growth. }
Calc. Var. Partial Differential Equations 59 (2020), no. 5, Paper No. 171, 32 pp.

\bibitem{DeFilippis-Mingione1}
C.~De Filippis, G. Mingione,
\textit{Lipschitz bounds and nonautonomous integrals. }
Arch. Rational Mech. Anal. 242 (2021): 973--1057.

\bibitem{DeFilippis-Mingione2}
C.~De Filippis, G. Mingione,
\textit{On the regularity of minima of non-autonomous functionals.}
J. Geom. Anal. 30 (2020): 1584--1626. 

\bibitem{DeFilippis-Piccinini}
C.~De Filippis, M.~Piccinini,
\textit{Borderline global regularity for nonuniformly elliptic systems. }
Int. Math. Res. Not. IMRN 2023, no. 20, 17324--17376.
  
\bibitem{DiBenedettobook} 
E.~DiBenedetto,
\textit{Degenerate parabolic equations. }
Universitext, Springer-Verlag, New York, 1993. 

\bibitem{fabes-kenig-serapioni}
E.B.~Fabes, C.E.~Kenig, R.P.~Serapioni, 
\textit{The local regularity of solutions of degenerate elliptic equations. }
Comm. Partial Differential Equations 7 (1982), no. 1, 77--116.

\bibitem{GPS}
F.~Giannetti, A.~Passarelli di Napoli, C.~Scheven, 
 \textit{On higher differentiability of solutions of parabolic systems with discontinuous coefficients and $(p,q)$-growth.} Proc. Royal Soc.  Edinburgh 150 (2020),  419--451.

\bibitem{Giusti}
E.~Giusti,
\textit{Direct Methods in the Calculus of Variations.}
World Scientific, Singapore, 2003.

\bibitem{Hirsch-Schaeffner}
  J.~Hirsch, M.~Sch\"affner,
  \textit{Growth conditions and regularity, an optimal local boundedness result.}
Commun. Contemp. Math. 23 (2021), no. 3, Paper No. 2050029, 17 pp.

\bibitem{Lichnewsky-Temam:1978}
A. Lichnewsky, R. Temam,
\textit{Pseudosolutions of the time-dependent minimal surface
problem.}
J. Differential Equations 30 (1978), no. 3, 340--364.

\bibitem{mar89}
P. Marcellini, \textit{ Regularity of minimizers of integrals of the calculus of variations with nonstandard growth conditions}. Arch. Ration. Mech. Anal. 105 (1989), 267--284.

\bibitem{mar91}
P. Marcellini, \textit{Regularity and existence of solutions of elliptic equations with p, q-growth conditions.} J. Differ. Equ. 90 (1991), 1--30.

\bibitem{mar20-2}P. Marcellini,  \textit{ Growth conditions and regularity for weak solutions to nonlinear elliptic PDEs.}
{ J. Math.  Anal.   Appl.} 501 (2021), 124408.

\bibitem{dark-survey}
G.~Mingione, 
\textit{Regularity of minima: an invitation to the dark side of the calculus of variations. }
Appl. Math. 51 (2006), no. 4, 355--426.

\bibitem{Mingione-Radulescu} G. Mingione, V. R\u{a}dulescu, \textit{Recent
developments in problems with nonstandard growth and nonuniform ellipticity.} 
J. Math. Anal. Appl. 501 (2021), no. 1, Paper No. 125197, 41 pp.

\bibitem{pingen}
M.~Pingen, \textit{Regularity results for degenerate elliptic systems. }
Ann. Inst. H. Poincar\'e C Anal. Non Lin\'eaire 25 (2008), no. 2, 369--380.

\bibitem{Showalter}
R.~Showalter,
\textit{Monotone Operators in Banach Space and
Nonlinear Partial Differential Equations.}
Am. Math. Soc., Providence, RI, 1997.

\bibitem{singer3}
T. Singer,
\textit{Existence of weak solutions of parabolic systems with $p,q$-growth. }
Manuscripta Math. 151 (2016), no. 1-2, 87--112.

\bibitem{singer2}
T. Singer,
\textit{Local boundedness of variational solutions to evolutionary problems with non-standard growth. }
NoDEA Nonlinear Differential Equations Appl. 23 (2016), no. 2, Art. 19, 23 pp.

\bibitem{singer1}
T. Singer,
\textit{Parabolic equations with $p,q$-growth: the subquadratic case. }
Q. J. Math. 66 (2015), no. 2, 707--742.

\bibitem{Strauss}
W.~Strauss, 
\textit{On continuity of functions with values in various Banach spaces. }
Pacific J. Math. 19 (1966), 543--551.

\bibitem{Trudinger:1971}
 N.S.~Trudinger,
 \textit{On the regularity of generalized solutions of linear, non-uniformly elliptic equations. }
Arch. Rational Mech. Anal. 42 (1971), 50--62.

\bibitem{trud2}
 N.S.~Trudinger,
 \textit{Linear elliptic operators with measurable coefficients. }
Ann. Scuola Norm. Sup. Pisa Cl. Sci. (3) 27 (1973), 265--308.

\end{thebibliography}
\end{document}